\documentclass[12pt]{amsart}
\usepackage[margin=1in]{geometry}
\usepackage{amsmath,amsfonts,graphicx}
\usepackage{amsthm}
\usepackage{amssymb}
\usepackage{enumerate}
\usepackage{tikz-cd}
\usepackage{mathtools}
\usepackage{float}
\usepackage{hyperref}
\usepackage[capitalise]{cleveref}
\usepackage[T2A,T1]{fontenc}
\usepackage[utf8]{inputenc}
\usepackage[russian,english]{babel}
\usepackage{lipsum}
\graphicspath{{./images/}}

\title{Orders on Free Metabelian Groups}
\author{Wenhao Wang}
\address{Department of Mathematical Logic\\
  The Steklov Mathematical Institute of Russian Academy of Science\\
 Moscow, Russia, 119991}
\email[W.~Wang]{wenhaowang@mi-ras.ru}

\newtheorem{theorem}{Theorem}[section]
\newtheorem{lemma}[theorem]{Lemma}
\newtheorem{corollary}[theorem]{Corollary}

\newtheorem{innercustomthm}{Theorem}
\newenvironment{customthm}[1]
  {\renewcommand\theinnercustomthm{#1}\innercustomthm}
  {\endinnercustomthm}

\theoremstyle{definition}
\newtheorem{proposition}[theorem]{Proposition}
\newtheorem{definition}[theorem]{Definition}

\theoremstyle{remark}

\newtheorem*{rems}{Remark}

\DeclareMathOperator{\supp}{supp}

\DeclareMathOperator{\CI}{CI}

\begin{document}

\maketitle

\begin{abstract}
	A bi-order on a group $G$ is a total, bi-multiplication invariant order. A subset $S$ in an ordered group $(G,\leqslant)$ is convex if for all $f\leqslant g$ in $S$, every element $h\in G$ satisfying $f\leqslant h \leqslant g$ belongs to $S$. In this paper, we show that the derived subgroup of the free metabelian group of rank 2 is convex with respect to any bi-order. Moreover, we study the convex hull of the derived subgroup of a free metabelian group of higher rank. As an application, we prove that the space of bi-order of non-abelian free metabelian group of finite rank is homeomorphic to the Cantor set. In addition, we show that no bi-order for these groups can be recognised by a regular language.
\end{abstract}

\section{Introduction}

A group $G$ is \emph{bi-orderable} if there exists a total order $\leqslant$ which is invariant under multiplication from both sides, i.e., if for $g,h\in G$ with $g\leqslant h$ then $f_1 g f_2\leqslant f_1 h f_2$ for all $f_1,f_2\in G$. Such a total order is called a bi-invariant order or bi-order for short on the group $G$. Similarly, a group $G$ is \emph{left-orderable} (\emph{right-orderable}) if there exists a left-invariant (right-invariant) order on $G$, in which the order is invariant under left-multiplication (right-multiplication). It is not hard to see that right-orders and left-orders have a one-to-one correspondence. Thus, in this paper, we will only discuss left-orders and bi-orders on a group. For every order $\leqslant$ on $G$, the positive cone $P_\leqslant$ consists of all positive elements in $G$ under $\leqslant$. It is a semigroup and $G=P\sqcup P^{-1}\sqcup\{1\}$. If $\leqslant$ is bi-invariant, then $P$ is invariant under conjugation. A positive cone completely determines the corresponding order, and vice versa. Hence, in this paper, we will identify positive cones and their associated orders when it is convenient. 

The free metabelian group of rank $n$ is the quotient of the free group of rank $n$ by its second derived subgroup, which processes the following presentation.
\[M_n=\langle a_{1},a_2,\dots,a_n \mid [[u,v],[w,z]]=1, \forall u,v,w,z\in \{a_1,a_2,\dots,a_n\}^{*}\rangle.\]
Recall that for any set $X$, the notation $X^*$ denotes the free monoid (including the empty word) generated by $X$ and $X^{-1}$. $M_n$ is bi-orderable since by Magnus embedding \cite{Magnus1939} it is a subgroup of $\mathbb Z^n\wr \mathbb Z^n$, which is bi-orderable as bi-orderability is closed under taking wreath products \cite[Theorem 2.1.1]{mura1977OrderableGroups}.

In this paper, we study the convex hull of the derived subgroup of a free metabelian group with respect to a bi-order, where the convex hull $\overline H$ of a subgroup $H$ is the smallest convex subgroup containing $H$. Let $M_n$ be the free metabelian group of rank $n$. We show that the derived subgroup is always convex when $n=2$.

\begin{customthm}{A}[\cref{rank2}]
\label{rank2Intro}
	$M_2'$ is convex with respect to any bi-invariant order on $M_2$.
\end{customthm}

When $n\geqslant 3$, we construct a bi-order such that $\overline M_n'\neq M_n'$ in \cref{notConvex}. But we can still obtain some information about the order from the restriction of the order on the derived subgroup.

\begin{customthm}{B}[\cref{higherRank}]
	\label{quotientRank2Intro}
		Let $\leqslant$	be a bi-invariant order on $M_n$ then the rank of $M_n/\overline M_n'$ is greater or equal to 2, where $\overline M_n'$ is the convex hull of the derived subgroup with respect to $\leqslant$. 
\end{customthm}

Let $\mathcal{LO}(G)$ be the set of all left-orders on $G$. It carries a natural topology whose sub-basis is the family of sets of the form $V_g=\{P_\leqslant\mid 1\leqslant g\}$ for $g\in G$. The space $\mathcal{LO}(G)$ is a closed subset of the Cantor set and is metrizable (See, for example, \cite{deroin2014GroupsOrders}, \cite{clay2016OrderedGroupsTopology}). And the space of all bi-orders $\mathcal O(G)$ is a closed subspace of $\mathcal{LO}(G)$. A lot has been known about the structure of these spaces. The space of bi-orders of a non-cyclic free abelian group is homeomorphic to the Cantor set \cite{sikora2004TopologySpacesOrderings} and the same holds true for a non-abelian free group \cite{mccleary1989FreeLattice}, \cite{dovhyi2023Topology}. For the Braid group $B_n, n\geqslant 3$, the space $\mathcal{LO}(B_n)$ is infinite and has isolated points \cite{dubrovina2001OnBraidGroups}. Tararin gave a complete classification of groups which have finite space of left-orders \cite[Proposition 5.2.1]{kopytov1996RightOrderedGroups}. Due to \cref{rank2Intro} and \cref{quotientRank2Intro} we have:

\begin{customthm}{C}[\cref{cantorSet}]
	The space $\mathcal O(M_n)$ is homeomorphic to the Cantor set for $n\geqslant 2$. 
\end{customthm} 

Note that the space of left orders of a free metabelian group of rank $2$ or higher is also a Cantor set \cite{rivas2016SpaceLeftOrdering}.

Let $X$ be a generating set of $G$. A language $\mathcal L$ over $X$ is a subset of the free monoid $X^{*}$. A language is \emph{regular} if it is accepted by a finite state automaton, and is \emph{context-free} if it is accepted by a pushdown machine. We refer to \cite{hopcroft1979IntroductionAutomataTheory} for the definitions of finite state automata and pushdown machines.  
 
An order is computable if there exists an algorithm deciding if $u\leqslant v$ in $G$ for any pair of words $u,v \in X^*$. An order is regular (context-free) if the positive cone can be recognised by a regular (context-free) language. Computability of left-orders and bi-orders has gained a lot of interest lately. Harrison-Trainor \cite{harrison-trainor2018LeftOrderableComputable} have shown that there exists a left-orderable group with solvable word problem but no computable left-orders, while Darbinyan \cite{darbinyan2020Computability} has constructed an example for the case of bi-orders. \v{S}uni\'{c} \cite{sunic2013ExplicitLeftOrders, sunic2013OrdersFreeGroups} showed that there exists a one-counter left-orders on the free groups and, later with Hermiller \cite{hermiller2017NoPostiveCone}, proved that left-orders on free products are never regular which implies that such left-orders constructed by \v{S}uni\'{c} are the computationally simplest orders in the sense of Chomsky hierarchy. Antol\'{i}n, Rivas and Su \cite{antolin2021regular} have studied regular and context-free left-orders on groups and have shown that the metabelian Baumslag-Solitar group $BS(1,q), |q|>1$ does not admit a regular bi-invariant order. 

Recall that a group is \emph{computably bi-orderable} if the group admits a computable bi-order. As a consequence of a more general \cref{notRegularConradian}, we have that:
\begin{customthm}{D}[\cref{computablyBiorderable},\cref{notRegularMetabelian}]
	\label{nonregularityIntro}
	Let $M_n$ be the free metabelian group of rank $n$. Then every $M_n$ is computably bi-orderable. Moreover, $M_n$ admits a regular bi-order if and only if $n=1$.
\end{customthm}
	When $n=2$, it can be shown that $M_2$ admits a context-free bi-order. It remains unknown if the same holds true for $n\geqslant 3$. 

The paper is organised in the following way. In Section 2 we introduce the notion of comparison index as in Section 3 we introduce $Q$-invariant orders for $\mathbb ZQ$-modules where $Q$ is a free abelian group. In Section 4 and 5 we study $Q$-invariant orders for varies cases. We analyse the convex hull of the derived subgroups of free metabelian groups and prove the main theorems in Section 6 and 7. In Section 8 we show that a free metabelian group of finite rank is computably bi-orderable, but the order is never regular unless the group is the infinite cyclic group.

\emph{Acknowledgements.} The author acknowledges Igor Lysenok for many inspirational discussions on this subject and Crist\'{o}bal Rivas for pointing out that \cref{notRegularMetabelian} holds for Conradian orders. The author is also grateful to the referee/s for his/her many useful comments which led to a refinement of the paper and some new results. The work was performed at the Steklov International Mathematical Center and supported by the Ministry of Science and Higher Education of the Russian Federation (agreement no. 075-15-2022-265).

\section{Comparison Index on Orderable Abelian Groups}

It is convenient to firstly make our convention throughout the paper. For elements $f,g$ in a group $G$, we set that $f^g=g^{-1}fg$ and $[f,g]=f^{-1}g^{-1}fg$.

An abelian group is bi-orderable if and only if it is torsion-free. Orders on abelian groups of finite $\mathbb R$-rank have been well-understood (See \cite[Section 10.2]{clay2016OrderedGroupsTopology}\cite{teh1961ConstructionOrdersAbelian}). In particular, consider a free abelian group $A$ of rank $k$ and fix a basis of it. As $A$ is canonically embedded into $\mathbb R^k$ as a lattice, every order on $A$ corresponds to a hyperplane passing through the origin, where the hyperplane separates elements in $A$ into positive elements, negative elements and the identity. If the intersection of the hyperplane and $A$ is non-empty, we also need to choose an order on the hyperplane itself. Such a hyperplane can be determined by its normal vector $\mathbf{n}$ which points to the positive side of the lattice, and an element $g$ in $A$ is positive if the dot product $g\cdot \mathbf{n}$ is positive.

To capture this structure of order on a free abelian group, we introduce the notion of comparison index. 

\begin{definition}
	Let $x,y$ be positive elements in an ordered abelian group $(A,\leqslant)$. Then we define the comparison index of $x,y$ with respect to $\leqslant$ as the following.
\[\CI(x,y;\leqslant)=\lim_{n \to \infty} -\frac{m(n)}{n}, m(n)=\min \{m\mid mx+ny\geqslant 0_A\},n>0.\]
If the limit or the minimal element do not exist, we will denote $\CI(x,y; P_{\leqslant})=\infty$. And we make the convention that $\CI(0,0)=1$. If there is no ambiguity with the order considered, we will use $\CI(x,y)$ instead of $\CI(x,y;\leqslant)$ to relax the notation. 
\end{definition}

Notice that $mx+ny\geqslant 0_A$ if and only if $-mx-ny\leqslant 0_A$, and hence $|m(n)+m(-n)|\leqslant 1$. Therefore, an alternative definition of the comparison index is 
\[\CI(x,y;\leqslant)=\lim_{n \to -\infty} -\frac{m(n)}{n}, m(n)=\min \{m\mid mx+ny\geqslant 0_A\},n<0.\]

We also define the absolute value of an element in an ordered group $(G,\leqslant)$ as 
\[|g|=\begin{cases}
	g  &\text{if }g>1_G,\\
	g^{-1} &\text{if }g\leqslant 1_G.
\end{cases}\]

Fix an order $\leqslant$ on $A$. Given $x,y\in A$, we define the vector $\mathbf{r}(x,y)$ as follows.

\[\mathbf{r}(x,y)=\begin{cases}
	(1,\CI(|x|,|y|)) & \text{if $\CI(|x|,|y|;\leqslant)\neq \infty$},\\
	(0,1) & \text{if $\CI(|x|,|y|;\leqslant)= \infty$}.
\end{cases}\]

The next proposition gives a geometric interpretation of the comparison index. 

\begin{proposition}[Geometric interpretation of $\CI$]
\label{geometricInterpretation}
	Let $(A,\leqslant)$ be an ordered abelian group and $x,y$ be two positive elements. We have 
	\begin{enumerate}[(i)]
		\item $\CI(x,y)\in [0,\infty)\cup\{\infty\}$.
		\item Suppose that $\langle x,y\rangle$ has rank $2$ and $\mathbf{r}=\mathbf{r}(x,y)$. Let $(\langle x,y\rangle,\leqslant)\to (\mathbb Z^2,\prec)$ be the order preserving isomorphism such that $x\mapsto (1,0)$ and $y\mapsto (0,1)$. The order of $(\mathbb Z^2,\prec)$ can be extended to $(\mathbb R^2,\prec')$. The vector $\mathbf{r}$ is the normal vector to the hyperplane defining $\prec$ and $(0,0)\prec' \mathbf{r}$. In particular,
			\[(m, n) \cdot \mathbf{r}>0 \Rightarrow mx+ny>0_A.\]
	\end{enumerate}
\end{proposition}

\begin{proof}
	\begin{enumerate}[(i)]
		\item Since $x,y$ are both positive, then so is $mx+ny$ for any $m \geqslant 0, n>0$. Thus, for a fixed $n>0$ the minimum $m(n)=\min \{m\mid mx+ny\geqslant 0_A\}$ is non-positive. Hence, we have that $-\frac{m(n)}{n}\geqslant 0$ for all $n>0$.
			
			If $x>ny$ for any $n\in \mathbb Z$, then 
			\[\CI(x,y;\leqslant)=\lim_{n \to \infty} -\frac{m(n)}{n}=\lim_{n \to \infty} -\frac{0}{n}=0.\]
			
			If $y>mx$ for any $m$, then the limit
			\[\lim_{n \to \infty} -\frac{m(n)}{n}\]
			does not exist, which implies that $\CI(x,y)=\infty$
		
			Therefore, $\CI(x,y)\in [0,\infty) \cup \{\infty\}$.
		\item We will abuse the notations $x$ and $y$, and denote the axes containing $(1,0)$ and $(0,1)$ by $x$-axis and $y$-axis, respectively, of plane $(\mathbb R^2,\prec')$. 
		
			First, we consider the case that $x<ny$ for some $n$.  Let $l:y=kx$ be the hyperplane defining $\prec'$ where $k<0$, since both $x$ and $y$ are positive. We want to show that $-k \cdot \CI(x,y)=1$.
			
			Note that since $(\langle x,y\rangle,\leqslant)$ and $(\mathbb Z^2,\prec)$ are isomorphic, where the isomorphism defined in the statement is order-preserving, we have
				\[mx+ny \geqslant 0_{A} \iff (m,n)\succeq 0_{\mathbb Z^2}.\]
			
			Since $l$ defines $\prec$, then for a fixed $n>0$,
				\[(\lfloor \frac{n}{k}\rfloor+1,n)\succeq 0_{\mathbb Z^2}, \text{ and } (\lfloor \frac{n}{k}\rfloor-1,n)\prec 0_{\mathbb Z^2}.\]
			Hence
				\[\lfloor \frac{n}{k}\rfloor -1<m(n)\leqslant \lfloor \frac{n}{k}\rfloor +1,\]
				where $m(n)=\min \{m\mid mx+ny\geqslant 0_A\}.$ Thus 
				\[\lim_{n\to \infty} -\frac{1}{n}(\lfloor \frac{n}{k}\rfloor +1)\leqslant \CI(x,y)\leqslant \lim_{n\to \infty} -\frac{1}{n}(\lfloor \frac{n}{k}\rfloor -1).\]
			Therefore, we have $-k \cdot \CI(x,y)=1$. Consequently, $\mathbf{r}\perp l$ and $\mathbf{r}\succ' 0_{\mathbb Z^2}$.
				
			If $x>ny$ for all $n\in \mathbb Z$, then $l:y=0$ is the hyperplane defining $\prec$. Thus, $\mathbf{r}=(1,0)\perp l$ and $\mathbf{r}\succ 0_{\mathbb Z^2}$. 
			
			The rest of the statement follows immediately. 
	\end{enumerate}
\end{proof}

With this geometric interpretation, we have that:

\begin{theorem}
\label{dotProduct}
	Let $(A,\leqslant)$ be an ordered abelian group and $x_1,x_2,\dots,x_n$ be positive elements such that $\langle x_1,x_2,\dots,x_n\rangle$ is an abelian group of rank $n$. Suppose $r_i=\CI(x_1,x_i)\neq \infty$ for $i=2,3,\dots,n$. Then
	\[(m_1,m_2,\dots,m_n) \cdot (1,r_2,\dots,r_n)>0 \Rightarrow \sum_{i=1}^n m_ix_i>0_A.\]
\end{theorem}

\begin{proof}
	Let $(\langle x_1,x_2,\dots,x_n\rangle,\leqslant)\to (\mathbb Z^n,\prec)\subset (\mathbb R^n,\prec')$ be the order-preserving isomorphism such that $x_i\mapsto (0,\dots,1,\dots,0)$ in which the $i$-th coordinate is 1, and $\prec'$ is extended by $\prec$. Let $H$ be the hyperplane defining the order $\prec'$ in the space $\mathbb R^n$. Then the intersection, denoted by $l_i$, of $H$ and the $\mathbb R$-subspace spanned by $\{x_1,x_i\}$ is perpendicular to $\mathbf{r_i}=(1,0,\dots,r_i,\dots,0)$ in which the $i$-th coordinate is $r_i$ for $i=2,3,\dots,n$. Therefore, $\mathbf{r}=(1,r_2,r_3,\dots,r_n)$ is perpendicular to all $l_i$'s for $i=2,\dots,n$ and hence to $H$. Since all $r_i$ are non-negative and all $x_i$ are positive, then $\mathbf{r}\succ' 0_{\mathbb R^n}.$ It follows that
	\[(m_1,m_2,\dots,m_n) \cdot \mathbf{r} >0 \Rightarrow (m_1,m_2,\dots,m_n)\succ 0_{\mathbb Z^n} \Rightarrow \sum_{i=1}^n m_ix_i>0_A.\] 
\end{proof}

\begin{rems}
	Since $g>0_A$ if and only if $-g<0_A$, then we have that
	\[(m_1,m_2,\dots,m_n) \cdot (1,r_2,\dots,r_n)<0 \Rightarrow \sum_{i=1}^n m_ix_i<0_A.\]
\end{rems}

The following proposition provides some useful properties of the comparison index.

\begin{proposition}
\label{propositionComparisonIndex}
	Let $(A,\leqslant)$ be an ordered abelian group and $x,y,z$ be positive elements. We have
	\begin{enumerate}[(i)]
		\item $\CI(x,x)=1$. More generally, if $\langle x,y\rangle$ is cyclic and $mx=ny$, then $\CI(x,y)=\frac{m}{n}.$
		\item $\CI(x,y)=1/\CI(y,x)$ with the convention $\frac{1}{0}=\infty$ and $\frac{1}{\infty}=0.$
		\item $\CI(x,y) \cdot \CI(y,z)=\CI(x,z)$ if $\{\CI(x,y),\CI(y,z)\}\neq\{0,\infty\}$.
		\item Let $Q$ be a group acting on $A$ by order-preserving isomorphisms, then for all $q\in Q$, $\CI(x,y)=\CI(q \cdot x,q \cdot y)$.
	\end{enumerate}
\end{proposition}

\begin{proof}
	\begin{enumerate}[(i)]
		\item For positive elements $x,y$ such that $-mx+ny=0$ we have $-kmx+kny=0, $ for all $k>0$. Thus 
			\[\CI(x,y)=\lim_{n\to \infty} -\frac{-km}{kn}=\frac{m}{n}.\]
		\item It follows directly from the geometric interpretation of comparison index and (i).
		\item If $\langle x,y,z\rangle$ has rank 1, then the result follows from (i). 

			If $\langle x,y,z\rangle$ has rank 2, WLOG, we assume that $z=ax+by$ for $a,b\in \mathbb Z$. For the case that $\CI(x,y)\in (0,\infty)$. It is enough to show that $\CI(x,z)=a+b\CI(x,y)$. We let
			\[m_y(n)=\min\{ m \mid mx+ny\geqslant 0_A\}, \text{ and } m_z(n)=\min\{ m \mid mx+nz\geqslant 0_A\}.\]
			Notice that 
			\[mx+nz=(m+an)x+bny.\]
			Thus 
			\[m_z(n)=m_y(bn)-an.\]
			It follows that 
			\[\CI(x,z)=\lim_{n\to \infty} -\frac{m_z(n)}{n}=\lim_{n\to \infty} -b\cdot \frac{m_y(bn)}{bn}+a=a+b\CI(x,y).\]
			Therefore, if $z=ax+by$ we have
			\begin{align*}
				\CI(x,y) \cdot \CI(y,z)=\CI(x,y) \cdot \CI(y,ax+by)&=\CI(x,y)\cdot (\frac{a}{\CI(x,y)}+b)\\
				&=a+b\CI(x,y)\\
				&=\CI(x,z).
			\end{align*}
			If $\CI(x,y)=0$, then $x\geqslant ny$ for all $n$. Since $z=ax+by$ is positive, we have that $a\geqslant 0$. By the assumption $\{\CI(x,y),\CI(y,z)\}\neq\{0,\infty\}$ we have $\CI(y,z)\neq \infty$, which implies that there exists $n$ such that $ny>ax+by$. It forces $a=0$. Hence, $\CI(x,z)=0$. By (i) we have that $\CI(y,z)=b$, then $\CI(x,y)\cdot \CI(y,z)=\CI(x,z)$. The case where $\CI(x,y)=\infty$ is similar.

		\item Since $Q$ acts by order-preserving isomorphism, we have
			\[m(n)=\min \{m \mid mx+ny\geqslant 0_A\}=\min \{m \mid m (q\cdot x)+n(q \cdot y)\geqslant 0_A\}.\]
			Therefore
			\[\CI(x,y)=\lim_{n\to \infty} -\frac{m(n)}{n}=\CI(q\cdot x,q\cdot y).\]
	\end{enumerate}
\end{proof}

\begin{definition}
	Let $(A,\leqslant)$ be an ordered abelian group and $x,y$ be non-trivial elements. We say that $x$ and $y$ are \emph{comparable with respect to $\leqslant$} if $\CI(|x|,|y|;\leqslant)\in (0,\infty)$, and we write $x \sim y$. We say that $x$ is \emph{lexicographically less} than $y$ with respect to $\leqslant$ if $\CI(|x|,|y|;\leqslant)=\infty$, and we write $x\ll y$. Furthermore, we say that $x$ is \emph{lexicographically greater} than $y$ with respect to $\leqslant$ if $\CI(|x|,|y|;\leqslant)=0$, and we write $x\gg y$. 
\end{definition}

Note that $0$ is lexicographically less than any non-trivial element. 

\begin{proposition}
	\label{inducedRelationPartialOrder}
	Let $(A,\leqslant)$ be an ordered abelian group.
	\begin{enumerate}[(i)]
		\item The relation $\sim$ (being comparable) is an equivalence relation.
		\item For all non-trivial elements $x,y\in A$, $x\sim y$ if and only if there exist $m,n\in \mathbb Z$ such that $x\leqslant ny$ and $y\leqslant mx$. 
		\item $x\ll y$ if and only $y\gg x$.
		\item $\ll$ defines a strict partial order on $A$.
		\item $\ll$ induces a strict total order on $A/\sim$.
	\end{enumerate}
\end{proposition}

\begin{proof}
	\begin{enumerate} [(i)]
		\item It follows from (i),(ii) and (iii) of \cref{propositionComparisonIndex}.
		\item  Let $\varepsilon,\eta\in \{-1,1\}$ be the unique elements such that $\varepsilon x,\eta y$ are positive. If $x\sim y$, then by \cref{geometricInterpretation}, we have that
			\[(m,n)\cdot  \mathbf{r}(\varepsilon x,\eta y)>0\Rightarrow m\varepsilon x+n \eta y>0_A.\]
			
			Thus, we pick $m,n$ such that $(m,-\eta)\cdot \mathbf{r}>0 $ and $(-\varepsilon,n)\cdot \mathbf{r}>0$. Then $x\leqslant ny$ and $y\leqslant mx$.
			
			The converse is straightforward.
		\item It follows from (ii) of \cref{propositionComparisonIndex}.
		\item It follows from (iii) of \cref{propositionComparisonIndex}.
		\item It follows from (i) and (iii).
	\end{enumerate}
\end{proof}

\section{$Q$-invariant orders on finitely generated free $\mathbb ZQ$-modules}

We begin with a study on orders on the group ring of free abelian group of finite rank, since the derived subgroup $M_2'$ is isomorphic to $\mathbb Z(x,y)$ (See \cite{bachmuth1965AutomorphismsFreeMetabelian}, \cite{Groves1986}). Note that since the order on $M_2'$ given by the restriction of an order on $M_2$ is compatible with not only the group operation but also the action of $Q:=M_2/M_2'$, we consider the following ordering structure on modules over the group ring of free abelian group of finite rank.

\begin{definition}
Let $Q$ be a free abelian group of finite rank and $M$ a finitely generated $\mathbb ZQ$-module. By a $Q$-invariant order $\leqslant$ on $M$ we mean an order satisfying:
\begin{enumerate}[(i)]
	\item if $m_1\leqslant m_2$ then $m_1+m_3\leqslant m_2+m_3$ for any $m_3\in M$;
	\item if $m_1\leqslant m_2$ then $q\cdot m_1\leqslant q\cdot m_2$ for any $q\in Q$.
\end{enumerate}	
If a $\mathbb ZQ$-module $M$ admits a $Q$-invariant order, we say that $M$ is \emph{$Q$-orderable}.
\end{definition}

Let $Q=\mathbb Z^n$ with basis $\{x_1,x_2,\dots,x_n\}$ and let $F_k$ be a free $\mathbb ZQ$-module of rank $k$ with basis $\{e_1,\dots,e_k\}$. $F_k$ is $Q$-orderable since $F_k$ embeds as the base group into the wreath product $\mathbb Z^k \wr \mathbb Z^n$, which is bi-orderable \cite[Theorem 2.1.1]{mura1977OrderableGroups}, and the inner automorphism action of $\mathbb Z^n$ on the base group is order-preserving.

We fix a $Q$-invariant order $\leqslant$ on $F_k$ and WLOG, we can assume that $e_1>e_2>\dots>e_n>0_{F_k}$. For each element $m\in F_k$ and $q,q'\in Q$ we define that $q \sim_m q'$ if $q \cdot m \sim q' \cdot m$ and $q\ll_m q'$ if $q \cdot m \ll q' \cdot m$. 

\begin{lemma}
\label{inducedOrderOnQuotients}
	Let $Q$ and $F_k$ as above and $m,m' \in F_k$.
	\begin{enumerate}[(i)]
		\item $\sim_m$ is an equivalence relation and $\ll_m$ is a partial order on $Q$.
		\item If $q_1 \sim_m q_2$, then $qq_1 \sim_m qq_2$ for all $q\in Q$. In particular, the set $Q_m=\{q\in Q \mid 1_Q \sim_m q\}$ is a subgroup of $Q$ and $\ll_m$ induces a bi-order on $Q/\sim_m$.
		\item If $m \sim m'$, then $Q_m=Q_{m'}$.
	\end{enumerate}
\end{lemma}

\begin{proof}
	\begin{enumerate}[(i)]
		\item It follows from \cref{inducedRelationPartialOrder}.
		\item By \cref{inducedRelationPartialOrder} (ii) we have that there exist $k,k'\in \mathbb Z_{>0}$ such that 
			\[q_1 \cdot m\leqslant k q_2\cdot m, \text{ and }q_2 \cdot m \leqslant k' q_1 \cdot m.\]
			Since $\leqslant$ is $Q$-invariant, we have
			\[qq_1 \cdot m\leqslant k qq_2\cdot m, \text{ and }qq_2 \cdot m \leqslant k' qq_1 \cdot m.\]
			Therefore, if $q_1 \sim_m q_2$, then $qq_1 \sim_m qq_2$ for all $q\in Q$. 
			Thus, $Q_m$ is a subgroup since $1_Q \sim_m q$ implies $q^{-1} \sim_m 1_Q$ and 
			\[1_Q \sim_m q_1\sim_m q_2 \Rightarrow q_1 \sim_m q_1^2 \sim_m q_1 q_2.\] 
			Consequently, $\ll_m$ induces a bi-order on $Q/\sim_m$.	
		\item If $m \sim m'$, then for $q\in Q_m$ and $q'\in Q_{m'}$, we have
			\[m \sim q\cdot m \sim q\cdot m' \sim q' \cdot m' \sim m'.\]
			The second equivalence comes from the fact that $\leqslant$ is $Q$-invariant. Therefore, $q\in Q_{m'}$	and $q'\in Q_m$.
	\end{enumerate}
\end{proof}

%

Let $(A,\leqslant)$ be an ordered abelian group. For a subset $S$ of $A$, we define: 
\[\mathrm{Max}(S)=\{s\in S \mid s\gg t \text{ or } s \sim t, \forall t\in S\}.\]

\begin{proposition}
\label{homomorphismToReals}
	Let $Q$ and $F_k$ as above. The basis $\{e_1,e_2,\dots, e_k\}$ of $F_k$ satisfies that $e_1>e_2>\dots>e_n>0_{F_k}$. Suppose that $Q_{e_1}=Q$. Then the homomorphism $\varphi:F_k\to \mathbb R$ given on the abelian generators $\{q \cdot e_i \mid q\in Q, i=1,2,\dots,k\}$ of $F_k$ by
	\[\varphi(q \cdot e_i)=\CI(e_1,q \cdot e_i)\]
	satisfies that $\varphi^{-1}((0,\infty))$ lies in the positive cone $P_{\leqslant}$
\end{proposition}

\begin{proof}
	Let $m\in F_k$. Let $B_1=\mathrm{Max}\{e_1,e_2,\dots,e_n\}$. We first notice that 
	\[\CI(e_1,q \cdot e_i)=0, \forall q\in Q, e_i\notin B_1.\]
	Since $e_1\in B_1$ and $e_i\notin B_1$, thus $e_1\gg e_i$, which implies that
	\[e_1\sim q \cdot e_1\gg q \cdot e_i.\]
	
	We want to show that if $\varphi(m)>0$, then $m>0_{F_k}$. Denote by $\supp m$ the support of $m$, where $\supp m$ consists of all abelian generators from $\{q\cdot e_i\mid q\in Q, i=1,2,\dots,k\}$ with nonzero coefficients.
	
	To simplify notation, the element $m$ can be uniquely written as
	\[m=a_1 m_1+a_2 m_2 +\dots+a_l m_l+m', a_i\in \mathbb Z, m_i\in \supp m\cap (Q\cdot B_1),\]
	and $m'$ lies in the abelian subgroup generated in $F_k$ by $\{\supp m\setminus (Q\cdot B_1)\}.$ Note that $\varphi(m')=0$. 
	Thus 
	\[\varphi(m)=a_1\CI(e_1,m_1)+a_2\CI(e_1,m_2)+\dots+a_l\CI(e_l,m_l).\]
	Consider the abelian subgroup $(A,\leqslant)$ of $(F_k,\leqslant)$ generated by $\{m_1,m_2,\dots,m_l\}.$ Notice that  
	\[\varphi(m)=(a_1,a_2,\dots,a_l)\cdot \CI(e_1,m_1)(1,r_2,r_3,\dots,r_{l}),\]
	where $r_i=\CI(m_1,m_i).$ Since $e_1\sim m_1, $ and hence $\CI(e_1,m_1)>0$, then by \cref{dotProduct} we have
	\[\varphi(m)>0 \Rightarrow (a_1,a_2,\dots,a_l)\cdot (1,r_2,r_3,\dots,r_{l}) \Rightarrow m>0_A=0_{F_k}.\]
\end{proof}

\section{The case where $Q$ is cyclic and $F_k$ has rank 1}

Let $(A,\leqslant)$ be an orderable abelian group and let $Q$ be a group acting on $A$ by order-preserving isomorphisms. Recall that a subgroup $H$ of $A$ is \emph{convex} if for any $g\in A$ satisfying that there exists $h_1,h_2\in H$ such that $h_1\leqslant g \leqslant h_2$, then $g\in H$. A subgroup $H$ is \emph{$Q$-convex} if $H$ is a convex subgroup and is invariant under the action of $Q$. Note that if $H_1, H_2$ are convex subgroups of $A$, then either $H_1\subset H_2$ or $H_2\subset H_1$.

For the case $Q$ is cyclic and $F_k$ has rank 1, we can identify $F_1$ with the Laurent polynomials $\mathbb Z(x)$, on which the cyclic group $\langle x \rangle$ acts.

\begin{lemma}
\label{singleVariableLemma}
	Let $\prec$ be an $\langle x \rangle$-invariant order on $\mathbb Z(x)$. Let $\varepsilon\in \{-1,1\}$ be the unique element such that $\varepsilon \succ 0$ and let $r=\CI(\varepsilon,\varepsilon x;\prec)$. Moreover, for $f(x) \in \mathbb Z(x)$ let $a_-, a_+$ be the coefficients of the terms of the lowest degree and highest degree of $f(x)$ respectively. We have:
	\begin{enumerate}[(i)]
		\item when $r=0$, $f(x)\succ 0$ if and only if $\varepsilon a_->0$,
		\item when $r=\infty$, then $f(x)\succ 0$ if and only if $\varepsilon a_+>0$,
		\item when $r\in (0,\infty)$, then $f(x)\succ 0$ if $\varepsilon f(r)>0$. 
	\end{enumerate}
\end{lemma}

\begin{proof}
	For the case $r=0$, then $x\ll 1$. Thus, we can write any element $f(x)$ as 
	\[f(x)=a_- x^{k}+r(x), \]
	where $k$ is the lowest degree of $f(x)$ and $r(x)$ is the rest of $f(x)$. We claim that $a_- x^{k}\gg r(x)$. Since $x\ll 1$, we have
	\[\varepsilon x^{i}\succ n x^{i+1}, \forall n\in \mathbb Z.\]
	Thus, inductively we have 
	\[\varepsilon x^k \succ n r(x), \forall n\in \mathbb Z\]
	since the lowest degree of $r(x)$ is greater than $i$. In particular, $x^i\gg r(x)$.	Notice that $a_- x^{k}\succ 0$ if and only if $\varepsilon a_->0$. If $\varepsilon a_->0$, then $a_- x^{k}\succ nr(x)$ for all $n\in \mathbb Z$. Thus, $a_- x^{k}\succ -r(x)$, and hence $f(x) \succ 0$. On the other hand, if $f(x) \succ 0$, then $a_- x^k\succ -r(x)$. Because $a_- x^{k} \gg r(x)$, it forces $a_- x^k \succ n r(x)$ for all $n\in \mathbb Z$. In particular, $a_- x^k \succ 0$. Thus, $\varepsilon a_->0$. This completes the proof of the case $r=0$. The case where $r=\infty$ is similar.
	
	Now we assume $r\in (0,\infty)$. WLOG, we also assume that $\varepsilon=1$. Then $\CI(1,x^n)=r^n$ for all $n\in \mathbb Z$. Suppose $f(x)=\sum_{i=s}^t a_i x^i$ where $a_s,a_t\neq 0$. Then 
	\[f(r)=(a_s,a_{s+1},\dots,a_t)\cdot (r^{s},r^{s+1},\dots,r^t)=(a_s,a_{s+1},\dots,a_t)\cdot \CI(1,x^s)(1,r,\dots,r^{t-s}).\]
	Note that $\CI(x^s,x^i)=r^{i-s}$ for $i=s+1,s+2,\dots,t$. Then by \cref{dotProduct} we have
	\[f(r)>0 \Rightarrow f(x)\succ 0.\]
\end{proof}

\begin{lemma}
\label{maximalConvexSubgroup}
	Let $\prec$ be an $\langle x \rangle$-invariant order on $\mathbb Z(x)$. Let $H$ be the maximal proper $\langle x \rangle$-convex subgroup of $\mathbb Z(x)$. 
	\begin{enumerate}[(i)]
		\item $H$ is trivial if and only if $r$ is either a positive transcendental number, $0$, or $\infty$.
		\item If $r$ is algebraic and $p(x)$ is the primitive irreducible polynomial of $r$ in $\mathbb Z[x]$, then $H=p(x)\mathbb Z(x)$. In particular, $H$ is isomorphic to $\mathbb Z(x)$.
	\end{enumerate}
\end{lemma}

\begin{proof}
	\begin{enumerate}[(i)]
		\item Suppose $r=0,$ or $\infty$ and $H$ is not trivial. We claim that, in both cases, $1 \in H$. Since $H$ is non-trivial, there exists $f(x)\neq 0\in H$. WLOG we can assume $f(x)\succ 0$ and $1\succ 0$. There exists $x^k$ such that $x^kf(x)\succ 1$ by \cref{singleVariableLemma} (i), (ii). The claim is proved.
			
			Thus, $x^n\in H$ for all $n\in \mathbb Z$ because $H$ is $\langle x \rangle$-invariant. Then by \cref{singleVariableLemma} again, $H=\mathbb Z(x)$ since for any $g(x)\in \mathbb Z$ there exists $x^i, x^j$ such that $-x^{i}\prec g(x)\prec x^j$. This leads to a contradiction. Thus, $H$ is trivial if $r=0$ or $\infty$. 
			
			Suppose $r$ is transcendental and $H$ is not trivial. Let $f(x)\in H$ be a non-trivial element in $H$. WLOG, we assume that $1\succ 0$. Since $r$ is transcendental, $g(r)\neq 0$ for every element in $\mathbb Z(x)$. Then for every $g(x)\in \mathbb Z(x)$ there exists $k\in \mathbb Z$ such that 
			\[-kf(r)<g(r)<kf(r).\]
			Therefore, $g(x)\in H$. Hence, we have a contradiction.
			
			Conversely, if $r$ is algebraic, we claim that $I=\{f(x) \mid f(r)=0\}$ is a non-trivial proper $\langle x \rangle$-convex subgroup. WLOG, we assume that $1\succ 0$. Since $r$ is algebraic, $I$ is non-trivial. Suppose $g(x)\in \mathbb Z(x)$ such that $g(r)\notin I$. If $g(r)> 0$, then $g(x)-f(x)\succ 0$ for all $f(x) \in I$ by \cref{singleVariableLemma} (iii). Thus, $g(x)\succ I$. Similarly, if $g(r)<0$, then $g(x)\prec I$. Thus, $I$ is convex. This completes the proof of (i).
		\item By the discussion above, $I=p(x)\mathbb Z(x)\subset H$. Suppose there exists $f(x)\in H\setminus I$. WLOG, we assume that $1\succ 0$. Then $f(r)\neq 0$ and for any $g(x)\in \mathbb Z(x)$ there exists $k\in \mathbb Z$ such that 
			\[-kf(r)<g(r)<kf(r).\]
			Thus, $g(x)\in H$, which implies $H=\mathbb Z(x)$. A contradiction.
	\end{enumerate}
\end{proof}

Combining the above lemmas, we have a classification of all $\langle x \rangle$-invariant orders on $\mathbb Z(x)$.

\begin{theorem}
\label{classificationSingleVariable}
	Let $\prec$ be an $\langle x \rangle$-invariant order on $\mathbb Z(x)$. We inductively define a sequence $(r_1,r_2,\dots,)$ in $\mathbb R_{\geqslant 0}\cup \{\infty\}$ and a sequence $(\varepsilon_1,\varepsilon_2,\dots)$ in $\{-1,1\}$ as follows. 
	
	Let $p_1(x)=1$ and suppose we have already defined $(r_1,\dots,r_{s-1})$, $(\varepsilon_1,\dots,\varepsilon_{s-1})$ and $p_1,\dots,p_s$. Let $\varepsilon_s$ be the unique element in $\{-1,1\}$ such that $\varepsilon_s\prod_{i=1}^s p_i(x)\succ 0$. Then let 
			\[r_s=\CI(\varepsilon_s\prod_{i=1}^s p_i(x),\varepsilon_s x\prod_{i=1}^s p_i(x)),\]
		where $p_{s+1}(x)$ is the primitive irreducible polynomial of $r_s$ in $\mathbb Z[x]$.
		
	Then the $\langle x \rangle$-invariant order $\prec$ is codified by $(r_1,r_2,\dots)$ and $(\varepsilon_1,\varepsilon_2,\dots)$ as follows.
	\begin{enumerate}[(i)]
		\item let $H_0=\mathbb Z(x)$ and $H_i=p_i(x)H_{i-1}$ for $i=1,2,\dots$. Then we have a sequence of nested subgroups 
			\[\dots \subset H_n \subset \dots \subset H_2 \subset H_1 \subset H_0=\mathbb Z(x),\]
			where $H_i$ is the maximal proper $\langle x \rangle$-convex subgroup of $H_{i-1}$ for $i=1,2,\dots$. The nested sequence is finite if and only if there exists $r_k$ that is either transcendental, 0 or $\infty$.
				
		\item For $f(x)\in H_{s}\setminus H_{s+1}$ let $a_-, a_+$ be the coefficient of the term of the lowest degree and highest degree of $f(x)$ respectively. Then $f(x)\succ 0$ if and only if 
			\[\begin{cases}
				\varepsilon_s \left(\prod_{i=1}^s\frac{1}{p_i(r_s)}\right) f(r_s)>0 & \text{ if $r_s\neq 0,\infty$},\\
				\varepsilon_s a_->0 & \text{ if $r_s=0$},\\
				\varepsilon_s a_+>0 & \text{ if $r_s=\infty$}.
			\end{cases}\]
	\end{enumerate}
\end{theorem}

\begin{proof}
	The base case is handled by \cref{singleVariableLemma} and \cref{maximalConvexSubgroup}.
	
	We define an isomorphism $\varphi: H_{n}\to \mathbb Z(x)$ such that 
	\[f(x) \mapsto \left(\prod_{i=1}^n\frac{1}{p_i(x)}\right) f(x)\]
	We have an induced order $\prec_{\varphi}$ on $\mathbb Z(x).$ Let $r=\CI(\varepsilon, \varepsilon x; \prec_{\varphi})$, where $\varepsilon\in \{-1,1\}$ is the unique element such that $\varepsilon\prec_{\varphi} 0$. Thus, the maximal proper $\langle x \rangle$-convex subgroup $H$ of $(\mathbb Z(x),\prec_{\varphi})$ is either trivial for the case where $r$ is transcendental, $0$ and $\infty$, or $H=p(x)\mathbb Z(x)$ where $p(x)$ is the primitive irreducible polynomial of $r$ in $\mathbb Z[x]$. Note that 
	\[r=\CI(\varepsilon, \varepsilon x; \prec_{\varphi})=\CI(\varepsilon_n\prod_{i=1}^n p_i(x),\varepsilon_n x\prod_{i=1}^n p_i(x); \prec)=r_n, \text{ and } \varepsilon=\varepsilon_n.\]
 	Moreover, $H_{n+1}=\varphi^{-1}(H)$. Therefore, $H_{n+1}$ is trivial if and only if $r_n$ is transcendental, $0$ or $\infty$. And $H_{n+1}=p(x)H_n=p_n(x)H_n$. The rest of the theorem follows immediately.
\end{proof}

Consider a $Q$-orderable $\mathbb ZQ$-module $A$. Let $Q\mathcal{O}(A)$ be the space of $Q$-invariant orders on $A$. Note that since $\mathcal O(A)$ is a compact Hausdorff space and $Q$ acts on $\mathcal O(A)$ by homeomorphism, then $Q\mathcal{O}(A)$ is a closed subspace of $\mathcal O(A)$. Hence, $Q\mathcal{O}(A)$ is also a closed subspace of the Cantor set.

\begin{corollary}
\label{singleCantorSet}
	The space of $\langle x \rangle$-invariant orders on $\mathbb Z(x)$ is a Cantor set.
\end{corollary}

\begin{proof}
	For a fixed order $\prec$ it is enough to show that for any $f_1,\dots,f_k\succ 0$, there exists a $\langle x\rangle$-invariant order $\prec'$ such that $P_{\prec}\neq P_{\prec'}$ and $f_1,\dots, f_k\succ' 0$, where $P_\prec$ and $P_{\prec'}$ are positive cones for $\prec$ and $\prec'$ respectively.
	
	By \cref{classificationSingleVariable}, we have a sequence of nested subgroups 
			\[\dots \subset H_n \subset \dots \subset H_2 \subset H_1 \subset H_0=\mathbb Z(x),\]
	where $H_i$ is the maximal proper $\langle x \rangle$-convex subgroup of $H_{i-1}$ for $i=1,2,\dots$. And $\prec$ is codified by $(r_1,r_2,\dots)$ and $(\varepsilon_1,\varepsilon_2,\dots)$.
	
	Suppose $f_1,\dots,f_k\in \mathbb Z(x)\setminus H_s$ where $H_s$ is non-trivial. Then let $\prec'$ be the order such that $\prec'=\prec$ for $\mathbb Z(x)\setminus H_s$ and $\prec'=-\prec$ for $H_s$. Then $\prec'$ is the order we are looking for. In this case, $\prec'$ is codified as $(r_1,r_2,\dots)$ and $(\varepsilon_1,\dots,-\varepsilon_{s+1},\dots)$, where $\prec$ and $\prec'$ differ only at $\varepsilon_{s+1}$.
	
	If $H_s$ is trivial, WLOG, we can assume that $s=1$ and $\varepsilon_1=1$. We claim that there exists a transcendental number $r_1'$ such that $r_1'\neq r_1$, and $f_i(r_1')>0$ for all $i=1,\dots,k$.
	
	For the case $r_1=\infty$, we can assume that each $f_{i}$ does not consist of any negative power of $x$, since $f(x)\succ 0 \iff x^n f(x)$ for any $n\in \mathbb Z$. Then the leading coefficient is positive for all $f_i(x)$. Thus 
	\[\lim_{t\to \infty} f_i(t)=\infty, \forall i=1,2,\dots,k.\] 
	Therefore, there exists a transcendental number $r_1'$ large enough such that $f_i(r_1')>0$ for all $i=1,2,\dots,k$. The case when $r=0$ is similar.
	
	For the case $r_1$ is transcendental, we again assume that $f_{i}$ does not consist of any negative power of $x$. Each $f_i(x)$ can be written as 
	\[f_i(x)=(x-r_1)g_i(x)+c_i,\]
	where $c_i=f_i(r_1), g_i(x)\in \mathbb R(x)$. By the assumption, $c_i>0$ for all $i$. Notice that  
	\[\lim_{s\to r_1} (x-r_1)g_i(x)=0, \forall i=1,2,\dots,k.\]
	Then there exists a number $\varepsilon$ small enough such that $\varepsilon g_i(r_1+\varepsilon)>-c_i$ for all $i$. Therefore, such $r_1'$ exists. 
	
	Let $\prec'$ be the order codified by $(r_1')$ and $(1)$. Since $r_1\neq r_1'$, there exists a rational number $\frac{m}{n}, m,n\in \mathbb N$, between $r_1$ and $r_1'$. WLOG, we assume that $r_1<\frac{m}{n}<r_1'$. Thus, $nx-m\in P_{\prec'}\setminus P_{\prec}$. And hence $\prec$ and $\prec'$ are different. Moreover, $f_i(x)\succ' 0$ for all $i=1,2,\dots, k$, since $f_i(r_1')>0$. Therefore, we finish the proof.
\end{proof}

Recall that an order is \emph{Archimedean} if for every pair of positive elements $f,g$ there exists a natural number $n$ such that $g<f^n$. One immediate observation is that: 
\begin{corollary}
	An $\langle x\rangle$-invariant order on $\mathbb Z(x)$ is Archimedean if and only if it is codified by $(r_1)$ and $(\varepsilon_1)$ such that $r_1$ is transcendental.
\end{corollary}

\section{The case when $Q=\mathbb Z^n, n>1$}

Let $Q=\mathbb Z^n$ for $n>1$. We fix a $Q$-invariant order $\leqslant$ on $\mathbb ZQ$. Let $Q_1=\{q \in Q_0 \mid 1_Q \sim_1 q\}$. Recall that $q \sim_1 q'$ if and only if $q \cdot 1_{\mathbb ZQ} \sim q' \cdot 1_{\mathbb ZQ}$. By \cref{inducedOrderOnQuotients}, $Q_1$ is a subgroup and the order $\leqslant$ induces the order $\ll$ on $Q/Q_1$. In particular, if $Q_1$ is a proper subgroup of $Q$, then $Q/Q_1$ is orderable hence torsion-free. 

We pick a transversal $\tilde Q$ of cosets $Q/Q_1$. An element $m\in \mathbb ZQ$ can be uniquely written as 
\[m = q_1 \cdot m_1 + q_2 \cdot m_2 + \dots +q_l \cdot m_l,\]
where $q_i\in \tilde Q, m_i\in \mathbb ZQ_1$ and $q_1 \gg q_2 \gg \dots \gg q_l$. Each $m_i$ is called the coefficient of $q_i$, and $m_1$ is called the leading coefficient of $m$, denoted by $\mathrm{LC}(m).$ 

\begin{proposition}
\label{multiVariable}
	Let $(\mathbb ZQ,\leqslant)$, and $Q_1$ as above. Suppose $1\succ 0$. For any transversal $\tilde Q$ of $Q/Q_1$ we have that for all $m\in \mathbb ZQ$
	\[\varphi(\mathrm{LC}(m))> 0 \Rightarrow m\succ 0, \]
	where $\varphi$ is the homomorphism $\varphi: \mathbb ZQ_1\to \mathbb R$ given by 
	\[\varphi(q')=\CI(1,q'), \forall q'\in Q_1. \]
\end{proposition}

\begin{proof}
	Let $m\in \mathbb ZQ$. The $m$ can be written as 
	\[m=\sum_{i=1}^l \sum_{j=1}^{t_i} a_{ij}q_{ij},\]
	where $q_{ij}\sim_1 q_{ij'}, a_{ij}\in \mathbb Z$. 
	
	We pick a transversal $\tilde Q$ of cosets $Q/Q_1$. Then 
	\[m=\sum_{i=1}^l q_i\sum_{j=1}^{t_i} a_{ij}q_i^{-1}q_{ij}, q_i\in \tilde Q, q_i \sim_1 q_{i1}.\]
	Since $q_{i} \sim_1 q_{i1}$, then $1\sim_1 q_i^{-1}q_{ij}$. Thus, $m_i:=\sum_{j=1}^{t_i} a_{ij}q_i^{-1}q_{ij}\in \mathbb ZQ_1$.
	
	We claim that the sign of $\varphi(m_i)$ does not depend on the choice of $\tilde Q$. Take $q_i'\in q_iQ_1$. Then the coefficient becomes 
	\[m_i':=\sum_{j=1}^{t_i} a_{ij}q_i'^{-1}q_{ij}.\]
	We have that
	\[\varphi(m_i)=\sum_{j=1}^{t_i} a_{ij}\CI(1,q_i^{-1}q_{ij})=\CI(q_i,q_i')\sum_{j=1}^{t_i} a_{ij}
	\CI(1,q_i'^{-1}q_{ij})=\CI(q_i,q_i')\varphi(m_i'),\]
	since by \cref{propositionComparisonIndex}
	\[\CI(q_i,q_i')\CI(1,q_i'^{-1}q_{ij})=\CI(q_i,q_i')\CI(q_i',q_{ij})=\CI(q_i,q_{ij})=\CI(1,q_i^{-1}q_{ij}).\]
	Because $q_i'\in q_iQ_1$, the comparison index $\CI(q_i,q_i')$ is a positive real number. Therefore, the signs of $\varphi(m_i)$ and $\varphi(m_i')$ are always the same.
	
	Next, we will show that if $\varphi(\mathrm{LC}(m))>0$ then $m\succ 0$. Since we have a decomposition of $m$ as 
	\[m=\sum_{i=1}^l q_i\sum_{j=1}^{t_i} a_{ij}q_i^{-1}q_{ij}, q_i\in \tilde Q, q_i \sim_1 q_{i1}.\]
	We assume $q_1\gg q_i$ for $i=2,3,\dots,l$. Consider the abelian subgroup generated by $\{q_{ij}\mid i=1,2,\dots,l, j=1,2,\dots,t_i\}.$ Note that $\CI(q_1,q_{11})\in (0,\infty)$ and $\CI(q_{11},q_{ij})=0$ for $i\geqslant 2$. Moreover, we have that 
	\[\varphi(\mathrm{LC}(m))=\CI(q_1,q_{11})\sum_{j=1}^{t_i} a_{1j}\CI(1,q_{11}^{-1}q_{1j})=\CI(q_1,q_{11})\sum_{i=1}^{l} \sum_{j=1}^{t_i} a_{ij}\CI(q_{11},q_{ij}).\]
	
	By \cref{dotProduct}, 
	\[\varphi(\mathrm{LC}(m))=\CI(q_1,q_{11})\sum_{i=1}^{l} \sum_{j=1}^{t_i} a_{ij}\CI(q_{11},q_{ij})>0 \Rightarrow m\succ 0.\]
\end{proof}

When the rank of the free $\mathbb ZQ$-module exceeds 1, the situation is much more complicated, as $Q_{e_1}$ can be different from $Q_{e_2}$ for different basis elements $e_1,e_2$. The following proposition is a variation of \cref{dotProduct}, and it is heavily used in the later part of the paper.

\begin{proposition}
\label{leadingTermDominates}
	Let $Q=\mathbb Z^n$ and $(F_k,\leqslant)$ be the free $\mathbb ZQ$-module of rank $k$, where the basis element $e_1,e_2,\dots, e_k$ of $F_k$ are all positive. Let $f \in F_k$. Then $f$ can be written as 
	\[f=a_1 f_1+a_2 f_2+\dots a_s f_s+f_r, a_i\in \mathbb Z, f_i\in \{q\cdot e_j\mid q\in Q\}, f_r\in F_k,\]
	where $\{f_1,f_2,\dots,f_s\}=\mathrm{Max}(\supp f)$ and $f_1\gg f_r$. Then we have
	\[(a_1,a_2,\dots,a_s)\cdot (1,r_2,\dots,r_s)>0 \Rightarrow f>0_{F_k},\]
	where $r_i=\CI(f_1,f_i)$ for $i=2,3,\dots,s$.
\end{proposition}

\begin{proof}
	It directly follows from \cref{dotProduct}.
\end{proof}

In the decomposition of $f$, the part $a_1 f_1+a_2 f_2+\dots a_s f_s $ is called the leading term of $f$.

\section{Bi-orders on the Free Metabelian Group of rank 2}

The following commutator formulas are used throughout the paper.

\begin{lemma}
\label{commutatorFormula}
	Let $G$ be a group. Then for $a,b,c\in G$ we have
	\[[a,bc]=[a,c][a,b]^c, \text{ and }[ab,c]=[a,c]^b[a,c].\]
	In addition, if $G$ is metabelian, we have
	\[[a^m,b^n]=[a,b]^{(1+a)^{m-1}(1+b)^{n-1}}, \forall m,n\in \mathbb N.\]
\end{lemma}

\begin{proof}
	The proof is straightforward.
\end{proof}

Now we are ready to prove our main theorem for the free metabelian group of rank 2.

\begin{theorem}
\label{rank2}
	$M_2'$ is convex with respect to any bi-invariant order on $M_2$.
\end{theorem}

\begin{proof}	
	Let $a,b$ be generators of $M_2$, and $Q=\mathbb Z^2\cong M_2/M_2'$. The quotient map from $M_2$ to $Q$ is denoted by $\pi$. Then $x=\pi(a)$ and $y=\pi(b)$ form a basis of $Q$. 

	Let $\leqslant$ be a bi-order on $M_2$. The restriction of $\leqslant$ on $M_2'$ gives a $Q$-invariant order on the free $\mathbb ZQ$-module of rank 1. By replacing $a,b$ by $a^{-1}$ and $b^{-1}$ if necessary, we can always assume that $[a,b]>1$ and $a>1$. We then have an isomorphism $\iota: M_2' \to \mathbb Z(x,y)$ such that 
	\[\iota([a,b])=1, \iota([a,b]^a)=x, \iota([a,b]^b)=y.\]
	The order $\leqslant$ on $M_2'$ induces a Q-invariant order $\prec$ on $\mathbb Z(x,y)$.
	
	To prove the theorem, it suffices us to show that $|a^{i}b^{j}|>M_2'$. We first claim that it is enough to prove the theorem for the case $a>M_2'$. For $a^{i}b^{j}$ where $i,j$ are coprime, there exists an automorphism of $M_2$ such that it sends $a^{i}b^{j}$ to $ga$ where $g\in M_2'$. Note that the image of a positive cone is again a positive cone (possibly for a different order). We denote the induced order by $\leqslant'$. We observe that $|ga|>' M_2'$ if and only if $|a|>' M_2'$. Thus, $|a^ib^j|>M_2'$ if and only if $|a| >' M_2'$. Moreover, if $|a^{i}b^{j}|>M_2'$ for coprime $i,j$ then for every positive integer $k$ we have $|a^{ki}b^{kj}|>M_2'$. In summary, if $|a|> M_2'$ under any bi-order, then so does $|a^ib^j|$. Hence, the claim is proved. 
	
	Let $S\subset \mathbb Z(x,y)$ be a set consisting of all elements which pre-images are less than $a$, i.e.,
	\[S=\{f(x,y)\in \mathbb Z(x,y) \mid \iota^{-1}(f)<a\}.\]
	Note that $S$ is convex by its definition.
	
	We have the following properties of $S$. If $f\in S$, then
	
	\begin{enumerate}
		\item $0\in S$;
		\item $x^n f\in S$ for all $n\in \mathbb Z$;
		\item $y^n f-(1+y)^{n-1}\in S$ for all $n\in \mathbb N$;
		\item $y^{-n}f+y^{-n}(1+y)^{n-1}\in S$ for all $n\in \mathbb N$;
		\item $f+(x-1)g\in S$ for all $g\in \mathbb Z(x,y)$.
	\end{enumerate}
	
	We will provide the proof of these properties after the proof of the theorem.
	
	Thus, to prove $a>M_2'$ it is enough to show that $S=\mathbb Z(x,y)$.
	
	Let $\CI(1,x;\prec)=r,\CI(1,y;\prec)=s$. We have three cases to consider. 
	
	The first case is that $r,s\in (0,\infty)$. Thus, by \cref{multiVariable} we have a homomorphism $\varphi: \mathbb Z(x,y)\to \mathbb R$, where $x\mapsto r, y\mapsto s$, such that $\varphi^{-1}((0,\infty))\subset P_{\prec}$.
	
	Using the properties we have
	
	\begin{align*}
		0\in S &\xrightarrow{\text{by (3)}} y^{-n}(1+y)^{n-1}\in S \xrightarrow{\text{by (2) for $n=1$}} y^{n-1}(1+y)^{n-1}-1\in S\\
	& \xrightarrow{\text{by (2) for $n-1$}} (1+y)^{n-1}-y^{n-1}-(1+y)^{n-2}.
	\end{align*}
	If $s<1$, we have
	\[\lim_{n\to \infty} \varphi(y^{-n}(1+y)^{n-1})=\infty.\]
	And if $s\geqslant 1$, we have
	\[\lim_{n\to \infty} \varphi((1+y)^{n-1}-y^{n-1}-(1+y)^{n-2})=\infty.\]
	In either cases, for any $t>0$ there exists $f\in S$ such that $\varphi(f)>t.$ Equivalently, for any $g\in \mathbb Z(x,y)$ there exists $f\in S$ such that $f\succ g$. Therefore, $S=\mathbb Z(x,y)$.
	
	The second case is $r=0$ or $r=\infty$. By \cref{multiVariable}, $Q_1=\{q \mid 1\sim_1 q\}$ is proper. Take a transversal $\tilde Q$ of $Q/Q_1$. Again by \cref{multiVariable}, we have a homomorphism $\varphi': \mathbb ZQ_1\to \mathbb R$ such that $\varphi'^{-1}((0,\infty))$ is positive.
	
	Let $g\in \mathbb Z(x,y)$ be a positive element. Then it can be expressed in the following form.
	\[g=q_1\cdot g_1+q_2 \cdot g_2+\dots +q_l \cdot g_l, q_i\in \tilde Q, g_i\in \mathbb ZQ_1,\]
	where $q_1\gg q_2 \gg \dots \gg g_l.$ Since $g$ is positive, $\varphi'(g_1)\geqslant 0$.
	We let 
	\[g'=\begin{cases}
		q_1\cdot (g_1+1)+q_2 \cdot g_2 + \dots +q_l \cdot g_l & \text{ if $r=\infty$,}\\
		-x^{-1}(q_1\cdot (g_1+1)+q_2 \cdot q_2 + \dots +q_l \cdot g_l) & \text{ if $r=0$.}
	\end{cases}\]
	By the construction (4), we have $(x-1)g'\in S$. Note that if $r=\infty$, we have $x\gg 1$ and hence $xq_1 \gg xq_i, xq_1\gg q_i$ for $i\geqslant 2$. If $r=0$, we have $1\gg x$, and hence $x^{-1}q_1 \gg x^{-1}q_i, x^{-1}q_1\gg q_i$. In either cases, 
	\[\mathrm{LC}((x-1)g'-g)=g_1+1.\]
	Then 
	\[\varphi'(\mathrm{LC}((x-1)g'-g))=\varphi'(g_1)+1>0.\]
	Therefore, $(x-1)g'\in S$ and $(x-1)g'>g$, which implies $g\in S$. Since the choice of $g$ is arbitrary, we have $S=\mathbb Z(x,y)$.
	
	The last case is that $r\in (0,\infty)$ and $s\in \{0,\infty\}$. In this case, $Q_1=\langle x\rangle$. By the discussion of the first case, we have already shown that 
	\[ y^{-n}(1+y)^{n-1}, (1+y)^{n-1}-y^{n-1}-(1+y)^{n-2} \in S, \forall n\in \mathbb N.\]
	Note that if $1\ll y$, we have 
	\[(1+y)^{n+1}-y^{n+1}-(1+y)^{n}\succ y^n \succ y^{-n} , \forall n\in \mathbb N.\]
	And if $1 \gg y$, we have
	\[ y^{-n}(1+y)^{n-1} \succ y^{-n} \succ y^{n}, \forall n\in \mathbb N.\]
	Since $S$ is convex, we have that $y^{-n},y^{n}\in S$ for all $n\in \mathbb N$. It immediately follows that $S=\mathbb Z(x,y).$
		
	Therefore, $a>\iota^{-1}(S)=M_2'$ in all cases. The theorem is proved.
\end{proof}

\begin{lemma}
	Let $S$ be as above. If $f\in S$, then	
	\begin{enumerate}
		\item $0\in S$;
		\item $x^n f\in S$ for all $n\in \mathbb Z$;
		\item $y^n f-(1+y)^{n-1}\in S$ for all $n\in \mathbb N$;
		\item $y^{-n}f+y^{-n}(1+y)^{n-1}\in S$ for all $n\in \mathbb N$;
		\item $f+(x-1)g\in S$ for all $g\in \mathbb Z(x,y)$.
	\end{enumerate}
\end{lemma}

\begin{proof}
	(1) is obvious, since $a$ is positive. 
	
	Let $f\in S$, we have 
	\[\iota^{-1}(f) <a \Rightarrow \iota^{-1}(f)^{a^n}<a, \forall n\in \mathbb Z.\]
	Thus 
	\[\iota(\iota^{-1}(f)^{a^n})=x^nf\in S.\]
	The construction (2) is proved.
	
	The proof of the rest constructions is similar.	Construction (3) and (4) follow from conjugating $\iota^{-1}(f) <a$ by $b^n,b^{-n}$ respectively for $n \in \mathbb N$. Construction (5) follows from conjugating  $\iota^{-1}(f) <a$ by $h\in M_2'$ where $\iota(h)=g$.
\end{proof}

It is well-known that left-orderability is preserved under group extensions. Let $G$ be an extension of $A$ by $Q$, where $\pi:G\to Q$ is the quotient map, and suppose $A,Q$ are left-orderable. In addition, if we assume $P_A$ and $P_Q$ are positive cones of $A$ and $Q$ respectively, then $P:=P_A\cup \pi^{-1}(P_T)$ is a positive cone of a left-order on $G$, and thus $G$ is also left-orderable. While in general, bi-orderability is not preserved under group extensions. But if we assume that $P_A$ is invariant under the action of $Q$, then in this case $P_A\cup \pi^{-1}(P_T)$ defines a bi-order on $G$. An order given by such a construction is called a \emph{lexicographical order leading by the quotient}.

\begin{proposition}
	\label{convexToLexicographical}
	Let $G$ be a finitely generated orderable group that is an extension of $A$ by $Q$. If $A$ is convex with respect to order $\leqslant$, then $\leqslant$ is a lexicographical order leading by the quotient where the order on the quotient $Q$ is induced by $\leqslant$.
\end{proposition}

\begin{proof}
	Let $\pi: G\to Q$ be the canonical quotient map. Then we define an order $\widetilde \leqslant$ on $Q$ in the following way: $q_1\ \widetilde \leqslant\ q_2$ in $Q$ if $\pi^{-1}(q_1)\leqslant \pi^{-1}(q_2)$ in $G$. It is well-defined, since $A$ is convex. Let $P_Q$ be the positive cone in $Q$ associated with $\widetilde \leqslant$ and $P_A$ the positive cone in $A$ associated with the restriction of $\leqslant$ on $A$. Then it is not hard to check that $P_A\cup \pi^{-1}(P_Q)$ is the positive cone associated with $\leqslant$ in G. Hence $\leqslant$ is a lexicographical order leading by the quotient.
\end{proof}

Thus, one immediate consequence of \cref{rank2} is the following.
\begin{corollary}
\label{unbounded}
	Any bi-invariant order $\leqslant$ on $M_2$ is a lexicographical order leading by the quotient with respect to the extension of $M_2'$ by $M_2/M_2'\cong \mathbb Z^2$.
\end{corollary}

\section{Bi-orders on Free Metabelian Groups of Higher Rank}

Since $M_n'$ is no longer a free $\mathbb ZQ$-module when $n>2$, we have to consider $Q$-invariant orders on general $Q$-orderable $\mathbb ZQ$-modules. The following lemma allows us to lift the $Q$-invariant order on a finitely generated $\mathbb ZQ$-module to a $Q$-invariant order on a free $\mathbb ZQ$-module.

\begin{lemma}
	\label{liftOrders}
	Let $M$ be a finitely generated $\mathbb ZQ$-module and $M\cong F_k/S$ where $F_k$ is a free $\mathbb ZQ$-module. If $M$ is $Q$-orderable, then for every $Q$-invariant order $\leqslant$ on $M$ there exists a $Q$-invariant order $\widetilde \leqslant$ such that $\widetilde \leqslant$ is the lexicographic order leading by the quotient with respect to $S$ and $(M,\leqslant)$. In particular, $S$ is convex under $\widetilde \leqslant$, and $\widetilde \leqslant$ is called a lift of $\leqslant$ to $F_k$. 
\end{lemma}

\begin{proof}
	Since $F_k$ is $Q$-orderable, so is $S$. We pick a $Q$-invariant order on $S$ and form a lexicographic order $\tilde \leqslant$ with respect to the order on $S$ and $(M,\leqslant)$. The rest is straightforward.
\end{proof}

Let $M_n$ be the free metabelian group of rank $n$ for $n>2$. Let $\{a_1,a_2,\dots,a_n\}$ be the canonical free generators of $M_n$ and $Q$ be the free abelian group with basis $\{x_1,x_2,\dots,x_n\}$. Then it is not hard to check that $M_n$ satisfies Jacobi's identities
\[[a_i,[a_j,a_k]][a_j,[a_k,a_i]][a_k,[a_i,a_j]]=1, \forall i,j,k\in \{1,2,\dots,n\}.\]
Let $D_n$ be a $\mathbb ZQ$-module with the module presentation
\begin{equation}
\label{presentationDn}
	D_n=\langle e_{ij}, 1\leqslant i<j\leqslant n \mid J(i,j,k)=0, \forall 1\leqslant i<j<k\leqslant n\rangle,\tag{$*$}
\end{equation}
where 
\[J(i,j,k)=(1-x_i)e_{jk}-(1-x_j)e_{ik}+(1-x_k)e_{ij}.\]

Let $\iota:M_n' \to D_n$ be a $\mathbb ZQ$-module homomorphism defined by 
\[\iota([a_i,a_j])=e_{ij}, \forall i<j, \iota([a_i,a_j]^{a_k})=x_k \cdot e_{ij}, \forall i<j, \forall k.\]
It has been shown that $\iota$ is an isomorphism (See \cite{bachmuth1965AutomorphismsFreeMetabelian}, \cite{Groves1986}).

Every bi-invariant order $\leqslant$ on $M_n$ induces a $Q$-invariant order $\prec$ on $D_n$. Let $F$ be the free $\mathbb ZQ$-module generated by $\{e_{ij} \mid 1\leqslant i<j\leqslant n\}$. In addition, let $J$ be the submodule generated by all Jacobi identities $\{J(i,j,k)\mid 1\leqslant i<j<k\leqslant n\}$. Thus, $D_n\cong F/J$ and the quotient map is denoted by $\rho: F\to D_n$. We lift $\prec$ to a $Q$-invariant order $\widetilde \prec$ on $F$ under which $J$ is convex.

For each element $g\in M_n$, let 
\[S_g=\{f\in D_n \mid \iota^{-1}(f)<|g|\}.\]

We denote by $B_1=\mathrm{Max}(\{e_{ij} \mid 1\leqslant i<j\leqslant n\})$ with respect to $\gg$. Recall that for a subset $S$ of an ordered abelian group $(A,\leqslant)$, we define 
\[\mathrm{Max}(S)=\{s\in S \mid s\gg t \text{ or } s \sim t, \forall t\in S\}.\]
\begin{lemma}
\label{crucialLemma01}
	Let $e_{ij}\in B_1$. If $Q_{e_{ij}}=Q$, Then 
	\[S_{a_i^{l}a_j^{m}}=D_n, \forall (l,m)\in \mathbb Z\times \mathbb Z-(0,0).\]
\end{lemma}

\begin{proof}
	We lift everything to the free $\mathbb ZQ$-module $(F,\widetilde \prec)$. It is enough to show that 
	\[\widetilde S_{a_i^l a_j^m}=\{f \in F \mid \iota^{-1}(\rho(f))<|a_i^l a_j^m|\}=F.\]
	
	Since $Q_{e_{ij}}=Q$, then $e_{ij} \sim x_k e_{ij}$ for all $k$. WLOG, we assume $e_{12}\in B_1$, $e_{12}\ \widetilde \succ\ 0$, and $Q_{e_{12}}=Q$. First, we notice that 
	\[e_{12} \gg qe_{ij}, \forall e_{ij}\notin B_1, q\in Q.\]
	Since $q e_{12} \sim e_{12}$ and $q e_{12}\gg q e_{ij}$.
	Let $\varepsilon_{ij} \in \{\pm 1\}$ be the unique element such that $\varepsilon_{ij} e_{ij}\succ 0$. Therefore, by \cref{homomorphismToReals} there exists a homomorphism $\varphi: F \to \mathbb R$, where 
	\[\varphi(q e_{ij})=\varepsilon_{ij}\CI(e_{12},\varepsilon_{ij} q e_{ij};\widetilde \prec),\forall q\in Q,\]
	such that $\varphi^{-1}((0,\infty))\subset P_{\widetilde \prec}$. 
	
	To simplify the proof, we will assume that $a_1^l a_2^m>1_{M_n}$ and $l>0$. Other cases are similar.
	
	Note that for $g\in M_n'$ 
	\[g<a_1^l a_2^m \iff g^{a_2^k} < (a_1^l a_2^m)^{a_2^k}=a_1^{la_2^k}a_2^m \iff g^{a_2^{-k}} < a_1^{la_2^{-k}}a_2^m, \forall k \in \mathbb N.\]
	Thus, if $f\in \widetilde S_{a_1}$, then 
	\begin{enumerate}
		\item $x_2^{k}f-(1+x_1)^{l-1}(1+x_2)^{k-1}e_{12}\in \widetilde S_{a_1},$ for $k\in \mathbb N$,
		\item $x_2^{-k}f+(1+x_1)^{l-1}x_2^{-k}(1+x_2)^{k-1}e_{12} \in \widetilde S_{a_1},$ for $k\in \mathbb N$.
	\end{enumerate}
	Similar to the proof of \cref{rank2}, we have
	\[(1+x_1)^{l-1}x_2^{-k}(1+x_2)^{k-1}e_{12}\in \widetilde S_{a_1}, (1+x_1)^{l-1}((1+x_2)^{k-1}-x_2^{k-1}-(1+x_2)^{k-2})e_{12} \in \widetilde S_{a_1}, \forall k\in \mathbb N.\]
	Thus, for any real number $t$, there exists $g\in \widetilde S_{a_1}$ such that $\varphi(g)>t$. Hence, $\widetilde S_{a_1^l a_2^m}=F$ for $(l,m)\neq (0,0)$. The lemma follows immediately.
\end{proof}

Let $\pi: M_n\to M_n/M_n' \cong Q$ be the canonical quotient map.
\begin{lemma}
\label{crucialLemma02}
	If $e_{ij}\in B_1$ and $q e_{ij}\not\sim e_{ij}$, then 
	\[S_g=D_n,\]
	where $\pi(g)=q$.
\end{lemma}

\begin{proof}
	Again, we lift everything to $(F,\tilde \prec)$. We assume $g>1_{M_n}$. Then for $h_1,h_2\in M_n'$, we have
	\[h_1<g \iff h_1^{h_2}< g^{h_2}.\]
	Thus, if $f\in \tilde S_{g}$ then 
	\[f+(1-q)h\in \tilde S_g, \forall h\in F.\]
	Here we use the property of metabelian groups that $g^{h_1}=g^{h_2}$ as long as $\pi(h_1)=\pi(h_2)$ in a metabelian group. Thus, $\iota(h_2^g)=q\cdot \iota(h_2).$
	
	Fix an arbitrary $h\in F$. WLOG, we suppose $e_{ij}\ \widetilde \succ \ 0$. Since $e_{ij}\in B_1$, there exists $q_0$ such that $q_0 e_{ij}\gg \supp{h}.$ If $q e_{ij}\gg e_{ij}$. Then we have
	\[(1-q)(-q_0 e_{ij})\in \widetilde S_g.\]
	Note that the leading term of $(1-q)(-q_0 e_{ij})-h$ is $qq_0 e_{ij}$. Thus, by \cref{leadingTermDominates},
	\[(1-q)(-q_0 e_{ij})-h\ \tilde \succ\  0 \Rightarrow (1-q)(-q_0 e_{ij}) \ \tilde \succ\  h \Rightarrow h\in \widetilde S_g.\]
	Similarly, if $e_{ij}\gg q e_{ij}$ we have
	\[(1-q)(q_0 e_{ij})\in \widetilde S_g,\]
	where the leading term of $(1-q)q_0 e_{ij}-h$ is $q_0 e_{ij}$. Thus, in this case $h\in \widetilde S_g$.
	
	Since the choice of $h$ is arbitrary, we have that $\tilde S_g=F$. 
\end{proof}

\begin{lemma}
\label{crucialLemma03}
	If $e_{ij}\in B_1$ and 
	\[x_i^{k}e_{ij} \text{ or }x_j^{k}e_{ij}\gg \mathrm{Max}\{x_k^{\pm 1} e_{ij} \mid k=1,2,\dots,n, e_{ij}\in B_1\}, \text{ for some } k\in \mathbb Z,\]
	then 
	\[S_{a_i^{l}a_j^{m}}=D_n, \forall (l,m)\in \mathbb Z\times \mathbb Z-(0,0).\]
\end{lemma}

\begin{proof}
	We lift everything to $(F,\tilde \prec)$ as above. WLOG, we assume $e_{12}\in B_1$ and $x_1^k e_{12}\gg \mathrm{Max}\{x_k^{\pm 1} e_{ij} \mid k=1,2,\dots,n, e_{ij}\in B_1\}.$ Since there exists $k$ such that $x_1^k e_{12}\gg \mathrm{Max}\{x_k^{\pm 1} e_{ij} \mid k=1,2,\dots,n, e_{ij}\in B_1\},$ then $Q_{e_{12}}\neq Q$. Thus, $S_{a_1^l}=D_n$ for all $l\in \mathbb Z-\{0\}$ by \cref{crucialLemma02}.
	
	The only remaining part is to prove $S_{a_1^l a_2^m}=D_n$ for $m\neq 0$. WLOG, we assume that $m>0$. As the discussion in \cref{crucialLemma01}, we have
	\[g\leqslant a_1^l a_2^m\iff g^{a_1^s}\leqslant a_1^la_2^{ma_1^s} \iff g^{a_1^{-s}}\leqslant a_1^la_2^{ma_1^{-s}}, \forall s\in \mathbb N.\]
	Thus, we have that if $f\in \widetilde S_{a_2}$, then
	\begin{enumerate}
		\item $x_1^s f+(1+x_1)^{s-1}(1+x_2)^{m-1}e_{12}\in \widetilde S_{a_2}$ for $s\in \mathbb N$,
		\item $x_1^{-s}f-x_1^{-s}(1+x_1)^{s-1}(1+x_2)^{m-1}e_{12}\in \widetilde S_{a_2}$ for $s\in \mathbb N$.
	\end{enumerate}
	Using those constructions, we can deduce that
	\begin{align*}
		0\in \widetilde S_{a_2}&\xrightarrow{\text{by (1)}} (1+x_1)^{s-1} (1+x_2)^{m-1} e_{12}\in \widetilde S_{a_2} \\
		&\xrightarrow{\text{by (2) for }s=1} (x_1^{-1}(1+x_1)^{s-1}-x_1^{-1}) (1+x_2)^{m-1} e_{12}\in \widetilde S_{a_2} \\
		&\xrightarrow{\text{by (2) for }s-1} (x_1^{-s}(1+x_1)^{s-1}-x_1^{-s}-x_1^{-(s-1)}(1+x_1)^{s-2}) (1+x_2)^{m-1} e_{12} \in \widetilde S_{a_2}.
	\end{align*}
	Since $(1+x_1)^{s-1} (1+x_2)^{m-1} e_{12}\in \widetilde S_{a_2},$ and $ (x_1^{-s}(1+x_1)^{s-1}-x_1^{-s}-x_1^{-(s-1)}(1+x_1)^{s-2}) (1+x_2)^{m-1} e_{12}\in \widetilde S_{a_2},$ then $x_1^{\pm s}e_{12}\in \widetilde S_{a_2}$ for all $s\in \mathbb Z$ no matter $x_2 e_{12} $ is comparable to $e_{12}$ or not. Since $x_1^k e_{12}\gg \mathrm{Max}\{x_k^{\pm 1} e_{ij} \mid k=1,2,\dots,n, e_{ij}\in B_1\}$ for some $k$ then for any choice of $q\in Q$ and $e_{ij}$ there exists $s\in \mathbb Z$ such that 
	\[q \cdot e_{ij}\ll x_1^s e_{12}\in \widetilde S_{a_1^la_2^m}.\]
	Therefore, $\widetilde S_{a_1^la_2^m}=F.$
\end{proof}

Now we are ready to prove the following theorem.
\begin{theorem}
\label{higherRank}
	For $n\geqslant 3$, the rank of $M_n/\overline{M_n'}$ is greater or equal to 2.
\end{theorem}

\begin{proof}
	We use the notation as above. The free metabelian group $M_n$ is generated by $\{a_1,a_2,\dots,a_n\}$, $M_n'$ is the derived subgroup of $M_n$, and $Q\cong \mathbb Z^n$ is the abelianization of $M_n$ with basis $\{x_1,x_2,\dots,x_n\}$. The canonical quotient map is denoted by $\pi:M_n\to Q$ where $\pi(a_i)=x_i$.
	
	Let $D_n$ be the $\mathbb ZQ$-module with presentation (\ref{presentationDn}). Moreover, $F$ is the free $\mathbb ZQ$-module with basis $\{e_{ij} \mid 1\leqslant i<j\leqslant n\}$ and $J$ is the submodule generated by Jacobi identities $\{J(i,j,k)\mid 1\leqslant i<j<k\leqslant n\}$. We denote by $\iota: (M_n',\leqslant)\to (D_n,\prec)$ the canonical isomorphism and by $\rho: (F,\tilde \prec)\to (D_n,\prec)$ the quotient homomorphism. And 
	\[S_g=\{f\in D_n \mid \iota^{-1}(f)<|g|\}\text{ and } \widetilde S_g= \{f\in F \mid \iota^{-1}\circ\rho(f)<|g|\}.\]
	The goal is to prove that there exist distinct $a_i,a_j$ such that 
	\[S_{a_i^{l}a_j^{m}}=D_n, \forall (l,m)\in \mathbb Z\times \mathbb Z-(0,0),\]
	or equivalently $|a_i^{l}a_j^{m}|>M_n'$ whenever $|a_i^{l}a_j^{m}|\neq 1_{M_n}$. It implies $a_i^l a_j^m>\overline M_n'$. Hence, $\mathbb Z^2=\langle x_i,x_j\rangle\subset M_n/\overline M_n'$.
	
	Let $B_1=\mathrm{Max}\{e_{ij}\mid 1\leqslant i<j\leqslant n\}$. WLOG, we assume $e_{12}\in B_1$ and $e_{12}\ \tilde \succ\  0_{F}$. If $Q_{e_{12}}=Q$, the result follows from \cref{crucialLemma01}. If not, then there exists $x_i e_{12} \not\sim e_{12}$. We have two cases.
	
	The first case is that 
	\[x_1^{k}e_{12} \text{ or } x_2^{k}e_{12}\gg T:=\mathrm{Max}\{x_l^{\pm 1} e_{ij} \mid l=1,2,\dots,n,e_{ij}\in B_1\} \text{ for some }k\in \mathbb Z.\]
	Then the result follows from \cref{crucialLemma03}. Note that if $x_1e_{12}\in T$, then $x_1^2 e_{12}\gg T$.
	
	Therefore, we only need to consider the case $x_i^{\varepsilon}e_{12}\in T$ for some $i>2, \varepsilon\in \{\pm 1\}$ and $x_i^{\varepsilon} e_{12} \gg x_1^l x_2^m e_{12}$ for all $(l,m)\in \mathbb Z\times \mathbb Z$.
	
	If $x_ie_{12}\in T$ for $i>2$, then we consider the Jacobi identity
	\[J(1,2,i)=(1-x_1)e_{2i}-(1-x_2)e_{1i}+(1-x_i)e_{12}.\]
	We notice that if $x_ie_{12}$ is lexicographically greater than every other term in $J(1,2,i)$, then $-x_ie_{12}$ is the leading term of $J(1,2,i)+e_{12}$. Thus, by \cref{leadingTermDominates}
	\[J(1,2,i)+e_{12} \ \tilde \prec\ 0_F \Rightarrow J(1,2,i) \ \tilde \prec\  -e_{12}\ \tilde \prec\ 0_F.\]
	It contradicts the fact that $J$ is convex under $\tilde \prec.$ Therefore $x_ie_{12}$ must be comparable to one of the other terms in $J(1,2,i)$. Since $Q_{e_{12}}\neq Q$ and $x_i e_{12}\in T$, then $x_ie_{12}\gg e_{12}$.
	
	If $x_ie_{12}\sim e_{1i}$ or $x_ie_{12}\sim e_{2i}$, then 
	\[e_{1i}\sim x_ie_{12}\gg e_{12} \text{ or }e_{2i} \gg e_{12},\]
	which contradicts $e_{12}\in B_1$. Thus, the only possible choice is $x_ie_{12}\sim x_1 e_{2i}$ or $x_ie_{12} \sim x_2 e_{1i}$. Consider the case where $x_i e_{12}\sim x_1 e_{2i}$. The other case is similar. The assumption implies that $e_{12}\sim e_{2i}$ otherwise $x_{1}e_{12}\gg x_1e_{2i}\sim x_ie_{12}$ contradicting to the fact that $x_i e_{12}\in T$. Therefore, $x_1 e_{12}\sim x_1 e_{2i} \sim x_i e_{12}$. It leads to a contradiction, since $x_1^{2}e_{12}\gg T$.
	
  If $x_i^{-1}e_{12} \in T$, then $e_{12}\gg x_i e_{12}$. As the discussion above, $e_{12}$ must be comparable with at least one of $e_{2i},x_1 e_{2i},e_{1i}$ and $x_2 e_{1i}$. 
  
	If $e_{12}\sim e_{2i}$, then $x_{i}^{-1}e_{12}\sim x_i^{-1}e_{2i}$. Thus, $x_{i}^{-1}e_{2i}\in T$ and hence by \cref{crucialLemma03} we have
	\[\tilde S_{a_2^la_i^m} = F, \forall (l,m)\in \mathbb Z\times \mathbb Z-(0,0).\]
	If $e_{12} \gg e_{2i}$ and $e_{12} \sim x_1 e_{2i}$, then $x_1 e_{12}\gg e_{12}$. By our assumption $x_i^{-1}e_{12}\gg x_1^k e_{12}$ for all $k\in \mathbb Z$. Therefore we have
	\[x_1^l x_i^{m} e_{12} \not\sim e_{12}, \forall (l,m)\in \mathbb Z\times \mathbb Z -(0,0)\]
	It follows from \cref{crucialLemma02} that 
	\[S_{a_i^{l}a_1^{m}}=D_n, \forall (l,m)\in \mathbb Z\times \mathbb Z-(0,0).\]
	Therefore, in this case, $\mathbb Z^{2}=\langle x_1,x_i\rangle\subset M_n/\overline M_n'$. This completes the proof.
\end{proof}

Combining \cref{rank2} and \cref{higherRank}, we immediately have:

\begin{corollary}
\label{cantorSet}
	The space $\mathcal O(M_n)$ is homeomorphic to the Cantor set for $n\geqslant 2$.
\end{corollary}

\begin{proof}
	Since the rank of the free abelian group $M_n'/\overline M_n'$ is greater or equal to 2, then every bi-order is a lexicographical order with respect to the extension of $\overline M_n'$ by $\mathbb Z^k$ for $k>1$. Since the space of orders of $\mathbb Z^k$ has no isolated point, so does $\mathcal O(M_n)$. 
\end{proof}

Next, we will show that $M_n'$ is not always convex when $n>2$.

\begin{lemma}
\label{constructOrder}
	Let $M$ be a $Q$-orderable $\mathbb ZQ$-module. Suppose there exists a homomorphism $\varphi:M\to \mathbb R$ such that if $\varphi(m)\neq 0$ then $\varphi(q\cdot m)/\varphi(m)>0$ for all $q\in Q$. Then there exists a $Q$-invariant order $\leqslant$ such that $\varphi^{-1}((0,\infty))\subset P_{\leqslant}$.
\end{lemma}

\begin{proof}
	Since $M$ is $Q$-orderable, so is $\ker \varphi$. Note that by the assumption, the positive cone $P=\{\varphi(m)\mid \varphi(m)>0\}$ of $\varphi(M)$ is $Q$-invariant, where the action of $Q$ on $\varphi(M)$ is defined by
	\[q\cdot \varphi(m):=\varphi(q\cdot m).\]
	 Thus, any lexicographic order $\leqslant$ formed by the extension of $\ker \varphi$ by $\varphi(M)$ is what we are looking for.
\end{proof}

With the above lemma, we are ready to prove the following.  

\begin{theorem}
\label{notConvex}
	Let $a_1^{t_1}a_2^{t_2}\dots a_n^{t_n}$ be a non-trivial element in $M_n$ for $n>2$. Then there exists a bi-invariant order such that $a_1^{t_1}a_2^{t_2}\dots a_n^{t_n}\in \overline M_n'$. In particular, for $n> 2$ there exists a bi-order on $M_n$ such that $M_n'$ is not convex.
\end{theorem}

\begin{proof}
	By utilizing the automorphism of $M_n$, the problem reduces to proving that there exists an order such that $a_1\in \overline M_n'$.
	
	We define a map $\varphi: M_n' \to \mathbb R$ as follows.
	\[\varphi([a_1,a_i])=0, \varphi([a_i,a_j]^{q})=1, \text{ for }i<j\in \{2,\dots,n\}, q\in Q.\]
	Since $\varphi$ sends all Jacobi identities to 0, $\varphi$ extents to a homomorphism from $M_n'$ to $\mathbb R$. Then by \cref{constructOrder}, there exists a $Q$-invariant order on $M_n'$ such that $\varphi^{-1}((0,\infty))\subset P_{\leqslant}$.

	Since $\varphi([a_1,a_i])=0$ for all $i=2,\dots,n$, by \cref{commutatorFormula} we have $\varphi([a_1^i,g])=0$ for all $i\in \mathbb Z, g\in M_n$. Moreover, $Q/\langle\bar a_1\rangle$ is a torsion-free finitely generated abelian group and hence bi-orderable, where $\bar a_1$ is the image of $a_1$ in $Q$. We denote $\pi: M_n\to Q/\langle\bar a_1\rangle$ to be the canonical quotient map, the composition of the quotient maps of $M_n\to Q$ and $Q\to \langle \bar a_1\rangle$.
	
	Let $P_{Q/\langle\bar q\rangle}$ be a positive cone on $Q/\langle\bar a_1\rangle$. We then let 
	\[P':=P_\leqslant \sqcup\left(\bigsqcup_{i\neq 0}\varphi^{-1}((0,\infty))a_1^{i}\right)\sqcup\left(\bigsqcup_{i>0} \ker \varphi a_1^i\right)\sqcup M_n'\pi^{-1}(P_{Q/\langle\bar a_1\rangle}).\]
	Let us check that $P'$ is a positive cone.
	
	Since $P_\leqslant$ and $P_{Q/\langle \bar a_1\rangle}$ are semigroups, $g_1g_2\in P'$ if $g_1,g_2\in P_\leqslant \sqcup M_n'\pi^{-1}(P_{Q/\langle\bar a_1\rangle})$. Note that 
	\[g_1 a_1^i g_2=g_1 g_2 [a_1^i,g_2]^{a_1^{-i}} a_1^i, \forall g_1,g_2\in M_n,\]
	and $\varphi([a_1^i,g_2]^{a_1^{-i}})=0$. Then it is not hard to check the product $g_1g_2\in P'$ where $g_1\in \left(\bigsqcup_{i\neq 0}\varphi^{-1}((0,\infty))a_1^{i}\right)\sqcup\left(\bigsqcup_{i>0} \ker \varphi a_1^i\right)$ and $g_2\in P_\leqslant \sqcup M_n'\pi^{-1}(P_{Q/\langle\bar a_1\rangle})$ since $\varphi(P_\leqslant)\geqslant 0$ and $\varphi(M_n')=0$. Therefore, $P'$ is a semigroup.
	
	The inverse of $P'$ can be written as follows.
	\[P'^{-1}=P_{\leqslant}^{-1} \sqcup\left(\bigsqcup_{i\neq 0}\varphi^{-1}((-\infty,0))a_1^{i}\right)\sqcup\left(\bigsqcup_{i<0} \ker \varphi a_1^i\right)\sqcup M_n'\pi^{-1}(P_{Q/\langle\bar a_1\rangle}^{-1}).\]
	Here we use the fact that 
	\[(ga_1^i)^{-1}=a_1^{-i}g^{-1}=[a_1^i,g]g^{-1}a_1^{-i}.\]
	Thus, we have $M_n=P'\sqcup {P'}^{-1}\sqcup\{1\}$.
	
	It is not hard to check that $P'$ is invariant under conjugation with elements in $M_n$. Therefore, $P'$ defines a bi-order $<'$ on the free metabelian group $M_n$. And by its definition, we immediately have
		\[\ker \varphi <' a_1<' \varphi^{-1}((0,\infty)).\]
	Therefore, $M_n'$ is not convex under $<'$ and any power of $a_1\in \overline M_n'$.
\end{proof}

\begin{rems}
When $n=2$, the map we define is a zero map. Thus, $M_2'=\ker \varphi$.
\end{rems}

\section{Orders on free metabelian groups are never regular}

Let $G$ be a finitely generated group and $X$ a finite generating set of $G$. An order $\leqslant$ on $G$ is said to be regular (context-free) if there exists a regular (context-free) language $\mathcal L\subset X^*$ such that $\pi(\mathcal L)=P_{\leqslant}$. An order $\leqslant$ is \emph{computable} if there exists an algorithm to decide if $g\leqslant h$ for any pair of $g,h\in G$. All those properties are independent of the choice of the finite generating set \cite[Lemma 2.11]{antolin2021regular}.

By \cref{rank2} and \cref{higherRank}, there are uncountably many bi-orders on $M_n$ for $n\geqslant 2$. Hence, there always exist uncountably many orders that are not computable on $M_n$.

In this section, we will show that there exist computable orders on $M_n$ but none of them are regular when $n\geqslant 2$.

Recall that by Magnus embedding (See \cite{Magnus1939}, \cite{Baumslag1973}), a free metabelian group of rank $n$ embeds into the wreath product of two free abelian groups of rank $n$. It naturally inherits a computable left-order from the wreath product \cite{antolin2021regular}. However, the regular lexicographical left-order on the wreath product where the base group leads is not bi-invariant. One workaround is to replace the lexicographical order by one which leads by the quotient. The order will become computably bi-invariant, but no longer regular.

Let $M_n$ be the free metabelian group of rank $n$ and $A_n,T_n$ free abelian groups of rank $n$. The generating sets of $M_n,A_n,T_n$ are respectively $X=\{x_1,x_2,\dots,x_n\}$, $A=\{a_1,a_2,\dots,a_n\}$ and $T=\{t_1,t_2,\dots,t_n\}$. The Magnus embedding $\varphi:M_n\to A_n\wr T_n$ is given by the homomorphism $\varphi(x_i)=a_it_i$.
	
Let $P_A$ and $P_T$ be regular positive cones of $A_n$ and $T_n$ respectively. We define a bi-order on the base group $B=\oplus_{t\in T_n}A_n$ as follows. Firstly, note that an element $f$ in $B$ can be uniquely written as a product of conjugates of elements in $A_n$ in the following fashion:
\[f=g_1^{h_1}g_2^{h_2}\dots g_s^{h_s}, g_i\in A_n,h_i\in T_n,\]
such that $h_1>h_2>\dots >h_s$ with respect to the order on $T_n.$ Thus we define $f>1$ if the leading term $g_1>1$. It is not hard to check that this order on $B$ is invariant under the action of $T_n$. The lexicographical order on $A_n \wr T_n$ is given by the positive cone $P=P_B\cup \pi^{-1}(P_T)$ where $\pi:A_n\wr T_n\to T_n$ is the canonical quotient map. 
	
One remark is that while the lexicographical order where the base group leads is regular, the new order we just define is not regular (it can be shown using the same idea as \cite[Lemma 3.11]{antolin2021regular}). Next we will show that $P$ can be recognised by a context-free language.
	
Let $\mathcal L_A, \mathcal L_T$ be the regular languages evaluate onto $P_A$ and $P_T$ respectively. We then define
\begin{align*}
	\mathcal L=&\{v u_1 w_1 u_2 w_2\dots u_n w_n z\mid u_1\in \mathcal L_A, w_i\in \mathcal L_T, u_2,\dots,u_n\in (A\cup A^{-1})^*,\\
	&v,z\in (T\cup T^{-1})^*,z=_{F(T)}(vw_1w_2\dots w_n)^{-1},n\neq 0\}\\
	&\sqcup \{v u_1 w_1 u_2 w_2\dots u_n w_n zz'\mid w_i,z'\in \mathcal L_T,u_i\in (A\cup A^{-1})^*,z=_{F(T)}(vw_1w_2\dots w_n)^{-1}\}.
\end{align*}
The first part of $\mathcal L$ covers the positive elements in the base group, and the second part covers the rest. The actual pushdown machine is not hard to construct, as it needs a stack to store the information on $vw_1w_2\dots w_n$. Thus, $\mathcal L$ is context-free.

The membership problem of $M_n$ in $A_n\wr T_n$ is not hard to solve. Therefore, $P\cap M_n$ is recursive at the very least. Whether the set $P\cap M_n$ can be recognised as a context-free language remains unknown. 

But for $M_2$ we can indeed construct a context-free bi-order on it. Let $X=\{a,b,c\}$, where $a,b$ generate $M_2$ and $c=[a,b]$. The quotient is generated by the set $T:=\{\bar a, \bar b\}$, where $\bar a,\bar b$ are the images of $a,b$ respectively. Let $P_Q$ be a regular positive cone on $Q=M_2/M_2'$ and $\mathcal L_Q$ be the corresponding regular language. We then define
\begin{align*}
	\mathcal L=&\{v c^{t_1} w_1 c^{t_2} w_2\dots c^{t_n} w_n z\mid t_1\in \mathbb N, w_i\in \mathcal L_Q, t_2,\dots,t_n\in \mathbb Z\setminus\{0\},\\
	&v,z\in (T\cup T^{-1})^*,z=_{F(T)}(vw_1w_2\dots w_n)^{-1},n\neq 0\}\\
	&\sqcup \{v c^{t_1} w_1 c^{t_2} w_2\dots c^{t_n} w_n zz'\mid w_i,z'\in \mathcal L_Q,t_i\in \mathbb Z\setminus\{0\},z=_{F(T)}(vw_1w_2\dots w_n)^{-1}\}.
\end{align*}
It is not hard to check $\mathcal L$ is context-free and recognises a bi-invariant positive cone of $M_2$.

In summary, we have that:
\begin{theorem}
	\label{computablyBiorderable}
	Free metabelian groups of finite rank are computably bi-orderable. Moreover, $M_1\cong \mathbb Z$ admits a regular bi-order and $M_2$ admits a context-free bi-order. 
\end{theorem}

Next we will show that none of the bi-orders are regular for $M_n, n\geqslant 2$ by proving a more general theorem for Conradian orders. Recall that a left-order is Conradian if for all positive elements $f,g$, there exists $n\in \mathbb N$ such that $g^{-1}f g^n$ is positive. A bi-order is always Conradian, since it is invariant under conjugation.

Recall that a subset $S$ of a metric space $(X,d)$ is \emph{coarsely connected} if there is $R>0$ such that the $R$-neighbourhood of $S$ is connected. The following lemma gives a description of a regular positive cone from a geometric perspective. 

\begin{lemma}[{\cite[Proposition 7.2]{alonso2022Geometry}}]
\label{coarselyConnected}
	Let $G$ be a finitely generated group. If $\leqslant$ is a regular order on $G$, then $P_{\leqslant}$ and $P_{\leqslant}^{-1}$ are coarsely connected subsets of the Cayley graph of $G$.
\end{lemma}

Next, recall that for a finitely generated group $G$ a non-trivial homomorphism $\varphi:G\to\mathbb R$ belongs to $\Sigma^1(G)$, the Bieri-Neumann-Strebel invariant (BNS invariant for short), if and only if $\varphi^{-1}((0,\infty))$ is coarsely connected. For details of BNS invariant, we refer \cite{bieri1987AGrometricInvariant}.

Following the idea of \cite[Lemma 3.11]{antolin2021regular}, we have

\begin{theorem}
\label{notRegularConradian}
	Let $G$ be a finitely generated Conradian orderable group with $\Sigma^1(G)=\emptyset$. Then no Conradian order of $G$ is regular.
\end{theorem}

\begin{proof}
	Let $G$ be a finitely generated Conradian orderable group with a Conradian order $\leqslant$. Consider the maximal proper subgroup $H\leqslant G$. Such $H$ exists and is normal \cite{mura1977OrderableGroups}. We have that $G/H$ is Archimedean with respect to the induced order. In particular, by H\"{o}lder's theorem $G/H$ is a free abelian group of finite rank \cite{holder1901}.

	Let $\pi:G\to G/H$ be the canonical quotient map. Since $H$ is convex, $P_{\leqslant}$ induces an order $\widetilde \leqslant$ on $G/H$. 
			
	We claim that there exists a homomorphism $\varphi:G/H \to \mathbb R$ such that $\varphi^{-1}((0,\infty))$ consists of positive elements with respect to $\widetilde \leqslant$. Since $G/H$ is a free abelian group of finite rank, then $\widetilde \leqslant$ corresponds to a hyperplane. Let $\mathbf{n}$ be a normal vector to the hyperplane, where $\mathbf{n}$ points to the positive side. Then the map $\varphi(g):=g\cdot \mathbf{n}$ is the homomorphism we are seeking for.
	
	Now let $f:=\varphi\circ \pi:G\to \mathbb R$. Note that $\ker \varphi$ is finitely generated as it is a subgroup of a free abelian group of finite rank. Let $P_1$ be the positive cone of $H$ and $g_1H,g_2H,\dots,g_{t}H$ be the generating set of $\ker \varphi$. Then we have
	\[P_{\leqslant}=f^{-1}((0,\infty))\cup P_1 \cup \left(  \bigcup_{g_1^{s_1}g_2^{s_2}\dots g_{t}^{s_{t}}H\ \widetilde>\ 1} g_1^{s_1}g_2^{s_2}\dots g_{t}^{s_{t}}H\right), s_i\in \mathbb Z.\]
	Since there exists a generator that is sent to a positive number, thus in the Cayley graph of $G$ the distance between $f^{-1}((0,\infty))$ and $P_1$ or any coset $g_1^{s_1}g_2^{s_2}\dots g_{t-1}^{s_{t-1}}H$ is 1. Thus, $f^{-1}((0,\infty))$ is coarsely connected if and only if $P_{\leqslant}$ is coarsely connected. By BNS theory, $f\in \Sigma^1(G)$ if and only $P_{\leqslant}$ is coarsely connected. Since $H$ contains the derived subgroup, we have 
	\[( g_1^{s_1}g_2^{s_2}\dots g_{t}^{s_{t}}H)^{-1}= g_1^{-s_1}g_2^{-s_2}\dots g_{t}^{-s_{t}}H.\]
	Thus,  $-f\in \Sigma^1(G)$ if and only $P_{\leqslant}^{-1}$ is coarsely connected. 
	
	If $P_{\leqslant}$ is a regular positive cone, both $P_{\leqslant} $ and $P_{\leqslant}^{-1}$ are coarsely connected by \cref{coarselyConnected}. Thus, $f$ and $-f$ belong to $\Sigma^1(G)$, which is a contradiction since $\Sigma^1(G)$ is empty.
\end{proof}

Note that for free metabelian group $M_n$ where $n\geqslant 2$, the centraliser of $M_n'$ in $\mathbb ZQ$ is trivial. Thus, $\Sigma^1(M_n)$ is empty \cite{bieri1980valuations}.

\begin{corollary}
	\label{notRegularMetabelian}
	The non-abelian free metabelian group of finite rank does not admit a regular Conradian order. In particular, no bi-order on a non-abelian free metabelian of finite rank is regular.
\end{corollary}
\bibliography{../MyLibrary}{}

\begin{thebibliography}{\v{S}13b}

\bibitem[AABR22]{alonso2022Geometry}
Juan Alonso, Yago Antol{\'\i}n, Joaquin Brum, and Crist{\'o}bal Rivas.
\newblock On the geometry of positive cones in finitely generated groups.
\newblock {\em Journal of the London Mathematical Society}, n/a(n/a), 2022.
\newblock \_eprint: https://onlinelibrary.wiley.com/doi/pdf/10.1112/jlms.12657.

\bibitem[ARS21]{antolin2021regular}
Yago Antolin, Cristobal Rivas, and Hang~Lu Su.
\newblock Regular left-orders on groups.
\newblock Avaliable at https://arxiv.org/abs/2104.04475, 2021.

\bibitem[Bac65]{bachmuth1965AutomorphismsFreeMetabelian}
S.~Bachmuth.
\newblock Automorphisms of free metabelian groups.
\newblock {\em Trans. Amer. Math. Soc.}, 118:93--104, 1965.

\bibitem[Bau73]{Baumslag1973}
Gilbert Baumslag.
\newblock Subgroups of finitely presented metabelian groups.
\newblock {\em J. Austral. Math. Soc.}, 16:98--110, 1973.
\newblock Collection of articles dedicated to the memory of Hanna Neumann, I.

\bibitem[BMR77]{mura1977OrderableGroups}
Roberta Botto~Mura and Akbar Rhemtulla.
\newblock {\em Orderable groups}, volume Vol. 27 of {\em Lecture Notes in Pure
  and Applied Mathematics}.
\newblock Marcel Dekker, Inc., New York-Basel, 1977.

\bibitem[BNS87]{bieri1987AGrometricInvariant}
Robert Bieri, Walter~D. Neumann, and Ralph Strebel.
\newblock A geometric invariant of discrete groups.
\newblock {\em Invent. Math.}, 90(3):451--477, 1987.

\bibitem[BS80]{bieri1980valuations}
Robert Bieri and Ralph Strebel.
\newblock Valuations and finitely presented metabelian groups.
\newblock {\em Proc. London Math. Soc. (3)}, 41(3):439--464, 1980.

\bibitem[CR16]{clay2016OrderedGroupsTopology}
Adam Clay and Dale Rolfsen.
\newblock {\em Ordered groups and topology}, volume 176 of {\em Graduate
  Studies in Mathematics}.
\newblock American Mathematical Society, Providence, RI, 2016.

\bibitem[Dar20]{darbinyan2020Computability}
Arman Darbinyan.
\newblock Computability, orders, and solvable groups.
\newblock {\em The Journal of Symbolic Logic}, 85(4):1588--1598, 2020.

\bibitem[DD01]{dubrovina2001OnBraidGroups}
T.~V. Dubrovina and N.~I. Dubrovin.
\newblock On braid groups.
\newblock {\em Mat. Sb.}, 192(5):53--64, 2001.

\bibitem[DM23]{dovhyi2023Topology}
Serhii Dovhyi and Kyrylo Muliarchyk.
\newblock On the topology of the space of bi-orderings of a free group on two
  generators.
\newblock {\em Groups Geom. Dyn.}, 17(2):613--632, 2023.

\bibitem[DNR14]{deroin2014GroupsOrders}
B.~Deroin, A.~Navas, and C.~Rivas.
\newblock Groups, orders, and dynamics.
\newblock Avaliable at https://arxiv.org/abs/1408.5805, 2014.

\bibitem[GM86]{Groves1986}
J.~R.~J. Groves and Charles~F. Miller, III.
\newblock Recognizing free metabelian groups.
\newblock {\em Illinois J. Math.}, 30(2):246--254, 1986.

\bibitem[H{\"o}l01]{holder1901}
O.~H{\"o}lder.
\newblock Die axiome der quantit{\"a}t und die lehre vom ma{\ss}.
\newblock {\em Ber. Verh. S{\"a}chs. Akad. Wiss. Leipzig Math. Phys. Kl.},
  53:1--64, 1901.

\bibitem[HT18]{harrison-trainor2018LeftOrderableComputable}
Matthew Harrison-Trainor.
\newblock Left-orderable computable groups.
\newblock {\em J. Symb. Log.}, 83(1):237--255, 2018.

\bibitem[HU79]{hopcroft1979IntroductionAutomataTheory}
John~E. Hopcroft and Jeffrey~D. Ullman.
\newblock {\em Introduction to automata theory, languages, and computation}.
\newblock Addison-Wesley Series in Computer Science. Addison-Wesley Publishing
  Co., Reading, Mass., 1979.

\bibitem[Hv17]{hermiller2017NoPostiveCone}
Susan Hermiller and Zoran \v{S}uni\'{c}.
\newblock No positive cone in a free product is regular.
\newblock {\em Internat. J. Algebra Comput.}, 27(8):1113--1120, 2017.

\bibitem[KM96]{kopytov1996RightOrderedGroups}
Valeri\u{\i}~M. Kopytov and Nikola\u{\i}~Ya. Medvedev.
\newblock {\em Right-ordered groups}.
\newblock Siberian School of Algebra and Logic. Consultants Bureau, New York,
  1996.

\bibitem[Mag39]{Magnus1939}
Wilhelm Magnus.
\newblock On a theorem of {M}arshall {H}all.
\newblock {\em Ann. of Math. (2)}, 40:764--768, 1939.

\bibitem[McC89]{mccleary1989FreeLattice}
Stephen~H. McCleary.
\newblock Free lattice-ordered groups.
\newblock In {\em Lattice-ordered groups}, volume~48 of {\em Math. Appl.},
  pages 206--227. Kluwer Acad. Publ., Dordrecht, 1989.

\bibitem[RT16]{rivas2016SpaceLeftOrdering}
Crist\'{o}bal Rivas and Romain Tessera.
\newblock On the space of left-orderings of virtually solvable groups.
\newblock {\em Groups Geom. Dyn.}, 10(1):65--90, 2016.

\bibitem[Sik04]{sikora2004TopologySpacesOrderings}
Adam~S. Sikora.
\newblock Topology on the spaces of orderings of groups.
\newblock {\em Bull. London Math. Soc.}, 36(4):519--526, 2004.

\bibitem[Teh61]{teh1961ConstructionOrdersAbelian}
H.-H. Teh.
\newblock Construction of orders in {A}belian groups.
\newblock {\em Proc. Cambridge Philos. Soc.}, 57:476--482, 1961.

\bibitem[\v{S}13a]{sunic2013ExplicitLeftOrders}
Zoran \v{S}uni\'{c}.
\newblock Explicit left orders on free groups extending the lexicographic order
  on free monoids.
\newblock {\em C. R. Math. Acad. Sci. Paris}, 351(13-14):507--511, 2013.

\bibitem[\v{S}13b]{sunic2013OrdersFreeGroups}
Zoran \v{S}uni\'{c}.
\newblock Orders on free groups induced by oriented words.
\newblock Avaliable at https://arxiv.org/abs/1309.6070, 2013.

\end{thebibliography}
\bibliographystyle{alpha}

\end{document}